\documentclass[12pt]{amsart}
\usepackage{amsmath,amsthm,amsfonts,amssymb,amscd,euscript}
\usepackage{color}
\usepackage{graphics}
\usepackage{tikz}
\input{xy}
\xyoption{all}
\usepackage{hyperref}
\usepackage{cleveref}
\usepackage{comment}
\usepackage{enumerate}
\newlength{\defbaselineskip}
\theoremstyle{plain}
\newtheorem{theorem}{Theorem}[section]

\newtheorem{lemma}[theorem]{Lemma}

\newtheorem{case}{Case}

\theoremstyle{definition}
\newtheorem{example}[theorem]{Example}
\newtheorem{definition}[theorem]{Definition}
\newtheorem{remark}[theorem]{Remark}

\begin{document}
	
	\title[Quiver Grassmannians of the Kronecker quiver]{An explicit description for quiver Grassmannians of the Kronecker quiver}
	\author{Ulysses alvarez }
	\address{Department of Mathematics, University of Alabama,
		Tuscaloosa, AL 35487, U.S.A. 
	}
	\email{uaalvarez@ua.edu}
	
	\author{Kyungyong Lee}
	\address{Department of Mathematics, University of Alabama,
		Tuscaloosa, AL 35487, U.S.A. 
		and Korea Institute for Advanced Study, Seoul 02455, Republic of Korea}
	\email{kyungyong.lee@ua.edu; klee1@kias.re.kr}
	
	%\date{October 2022}

	\begin{abstract}
		It is an open problem to find cell decompositions of quiver Grassmannians associated to each cluster variable. We initiate a new approach to this problem by giving an explicit description for each individual subrepresentation. In this paper, we illustrate our approach for the Kronecker quiver.
		In \cite{ALLS}, this approach will play a crucial role in studying cell decompositions of quiver Grassmannians for the Markov quiver.
	\end{abstract}
	
	\thanks{We were supported by the University of Alabama, Seoul National University, Korea Institute for Advanced Study, and the NSF grants DMS 2042786 and DMS 2302620. }
	
	\maketitle
	
	\section{Introduction}

	Cluster algebras have been introduced by Fomin and Zelevinsky in \cite{FZ} in the context of total positivity and canonical bases in Lie theory. Since then cluster algebras have been shown to be related to various fields in mathematics including representation theory of finite dimensional algebras, Teichm\"uller theory, Poisson geometry, combinatorics, Lie theory, tropical geometry and mathematical physics.
	
	A (skew-symmetric) cluster algebra is a subalgebra of a field of rational functions, given by specifying a set of generators, the so-called \emph{cluster variables}. These generators are constructed in a recursive way, starting from the initial variables $x_1,x_2,\ldots,x_N$, by a procedure called \emph{mutation}, which is determined by the choice of a skew-symmetric $N\times N$ integer matrix $B$ or, equivalently, by a quiver $Q$.

	%We mention two important results about cluster algebras.
	Any cluster variable is a Laurent polynomial in an initial cluster with integer coefficients \cite{FZ}, and moreover these coefficients are positive.

	\begin{theorem}[\cite{LS4,GHKK,Dav}]\label{positivity}
		The coefficients of any cluster variable in the Laurent expansion are positive integers.
	\end{theorem}

	\subsection{Cell decompositions of quiver Grassmannians}
	For a given quiver $Q$, a \emph{quiver representation} $M$ over $Q$ is a collection of vector spaces attached to vertices and linear maps attached to arrows. The collection of the dimensions of the vector spaces is called the \emph{dimension vector} of $M$. 
	For a given quiver representation $M$ over $Q$, a \emph{quiver Grassmannian} $\text{Gr}(M)$ (resp. $\text{Gr}_{\mathbf{e}}(M)$)  is a projective variety parametrizing subrepresentations of $M$ (resp. with the dimension vector $\mathbf{e}$).
	
	Derksen--Weyman--Zelevinsky showed the following:
	
	\begin{theorem}[\cite{DWZ2}]\label{dwz}
		Each coefficient in the Laurent expression of every cluster variable is equal to the Euler characteristic of (the union of) the corresponding quiver Grassmannian(s).
	\end{theorem}
	
	This theorem and the Positivity Theorem (Theorem~\ref{positivity}) together imply that the Euler characteristics of quiver Grassmannians coming from cluster variables are positive. When the Euler characteristic of a topological object, which naturally arises in representation theory, is positive, it is commonly expected that the object admits a cell decomposition.
	It is an important open problem to find a cell decomposition of each quiver Grassmannian coming from a cluster variable, if it exists.
	This problem has been answered for Dynkin quivers and tame quivers of types $\tilde{A}$ or $\tilde{D}$  (see \cite{CiE,LoWe} and references therein),  and rank 2 quivers \cite{RuWe}.  To give an answer to this problem for other quivers, we initiate a new approach by giving an explicit description for each and every individual subrepresentation. In this paper, we illustrate our approach for the Kronecker quiver. In \cite{ALLS}, this approach will play a crucial role in studying cell decompositions of quiver Grassmannians for the Markov quiver.
	
	\noindent\textbf{Acknowledgements.} We thank Li Li and Ralf Schiffler for very helpful discussions.

	\section{Quiver Grassmannians}
	Let $Q$ be the \emph{Kronecker quiver}, that is,
	$$\begin{tikzpicture}
		\node at (0,0) {$2$};
		\draw[->]{(0.5,0.1)--(1.5,0.1)};
		\node at (2,0) {$1$};
		\draw[->]{(0.5,-0.1)--(1.5,-0.1)};
	\end{tikzpicture}.$$
	The \textit{cluster algebra} $\mathcal{A}(Q)$ associated to $Q$ is the $\mathbb{Z}$-subalgebra of
	the ambient field $\mathbb{Q}(x_1,x_2)$ generated by the \textit{cluster variables}
	$x_m$ $(m \in \mathbb{Z})$ defined recursively by the relations
	\begin{center}
		$x_{m-1}x_{m+1}=x_m^2+1.$
	\end{center}
	Call $\{x_1,x_2\}$ the \textit{initial cluster} of $\mathcal{A}(Q)$. 
	
	Throughout the paper, for two integers $i_1,i_2\in \mathbb{Z}$, let $[i_1,i_2]=\{i\in \mathbb{Z}\, :\,  i_1\le i\le i_2  \}$.  
	
	Given $Q$ as described above, a (finite-dimensional) \textit{representation} $M$ over the field $\mathbb{C}$ is given by assigning 
	a finite dimensional $\mathbb{C}$-vector  space $M_i$ to each vertex $i \in [1,2]$ and linear maps $\phi_1,\phi_2\colon M_2 \to M_1$
	to every arrow from 2 to 1. Let $\text{rep}(Q)$ be the set of all representations of $Q$. 
	The \textit{dimension vector} (or  \textit{dimension} for short) of $M$ is the vector $(d_1,d_2)$ where the dimension of $M_i$ is $d_i$.
	A \textit{subrepresentation} $N$ of $M$ is a collection of subspaces $N_i \subseteq M_i$ such that
	$\phi_j(N_2) \subseteq N_1$ for all $i,j \in [1,2]$.
	For $\mathbf{e}=(e_1,e_2)\in \mathbb{Z}_{\ge 0}^2$, 
	the \textit{quiver Grassmannian} Gr$_\mathbf{e}(M)$ is defined to be the variety of all subrepresentations of $M$ with dimension $\mathbf{e}$.
	
	Let $\chi_\mathbf{e}(M)$ denote the \textit{Euler-Poincar\'e characteristic}
	of Gr$_\mathbf{e}(M)$.
	As in \cite{CZ}, we associate to any representation $M\in \text{rep}(Q)$ with dimension vector $(d_1,d_2)$
	a Laurent polynomial $X_M(x_1,x_2)\in\mathbb{Q}(x_1,x_2)$ given by
	$$X_M(x_1,x_2)=x_1^{-d_1}x_2^{-d_2} 
	\sum\limits_{\mathbf{e}=(e_1,e_2)\in \mathbb{Z}_{\ge 0}^2} \chi_\mathbf{e}(M)x_1^{2(d_2-e_2)}x_2^{2e_1}.$$
	
	For $m \in \mathbb{Z}_{\ge 3}$, it is easy to see that the denominator vector of the cluster variable $x_m$ is equal to $(m-2,m-3) \in \mathbb{Z}^2$. Let $\mathbf{I}_d$ be the $d \times d$ identity matrix for any $d\in \mathbb{Z}_{\ge 0}$.
	Let $M(m)\in \text{rep}(Q)$ be  the preprojective\footnote{In this paper, we focus on preprojective representations. For preinjective representations, one can employ a strategy similiar to the one given here.} representation 
	\begin{center}
		\begin{tikzpicture}
			\node at (-.3,0) {$\mathbb{C}^{m-3}$};
			\draw[->]{(0.5,0.1)--(1.7,0.1)};
			\node at (2.5,0) {$\mathbb{C}^{m-2}$};
			\node at (1.1,.7) {$\phi_1$};
			\draw[->]{(0.5,-0.1)--(1.7,-0.1)};
			\node at (1.1,-.7) {$\phi_2$};
		\end{tikzpicture}
	\end{center}
	where
	$$\aligned
	\phi_1((z_1,...,z_{m-3}))&=(z_1,\dots,z_{m-3})\left[\mathbf{I}_{m-3} | \mathbf{0}\right];\\
	\phi_2((z_1,...,z_{m-3}))&=(z_1,\dots,z_{m-3})\left[\mathbf{0} | \mathbf{I}_{m-3} \right].
	\endaligned$$
	Here $\mathbf{0}$ is the $(m-3)\times 1$ zero matrix.

	For each $n\in \mathbb{Z}_{\ge 0}$, let 
	$$
	\mathcal{I}_n = \{(i_1,\dots,i_{2n})\in [0,m-3]^{2n} \, : \, i_1\le i_2 < i_3\le i_4 < \cdots < i_{2n-1} \le i_{2n}\}
	$$
	and
	$$
	\mathcal{I}=\cup_{n\ge 0} \mathcal{I}_n.
	$$
	Caldero and Zelevinsky \cite{CZ} showed the following:
	\begin{theorem}\label{cz}
		For every $m\in \mathbb{Z}_{\ge 3}$, the cluster variable $x_m$ is equal to $X_{M(m)}(x_1,x_2)$. More precisely, 
		%\begin{equation}\label{chi}
		$$   \chi_{(e_1,e_2)}(M(m))=|\{(i_1,...,i_{2n})\in\mathcal{I}\, :\, e_1=n+\sum_{j=1}^n (i_{2j}-i_{2j-1})\text{ and }e_2=\sum_{j=1}^n (i_{2j}-i_{2j-1}) \}|. $$
		%\end{equation}
		Moreover, for each $P\in \mathcal{I}$ with $e_1=n+\sum_{j=1}^n (i_{2j}-i_{2j-1})\text{ and }e_2=\sum_{j=1}^n (i_{2j}-i_{2j-1})$, there exists a corresponding cell on  $\text{Gr}_{(e_1,e_2)}$.
	\end{theorem}
	
	In this paper, we  give an explicit description for each individual subrepresentation of $M(m)$ (see Remark~\ref{remark20240119}). As a corollary, we recover Theorem~\ref{cz}.

	\section{The matrices $N_2(P)$ and $N_1(P)$}\label{N2PN1P}
	In what follows, fix $m\in \mathbb{Z}_{\ge 3}$ and $P=(i_1,\dots,i_{2n})\in\mathcal{I}$. Let $e_1=e_1(P)=n+\sum_{j=1}^n (i_{2j}-i_{2j-1})$ and $e_2=e_2(P)=\sum_{j=1}^n (i_{2j}-i_{2j-1})$. 
	In this section we define matrices $N^{(2)}(P)$ and $N^{(1)}(P)$, as the first step of constructing all subrepresentations on the cell of  $\text{Gr}_{(e_1,e_2)}M(m)$ corresponding to $P$. The row spaces of these matrices will be used to  determine an $e_2$-dimensional  subspace of $\mathbb{C}^{m-3}$ and an $e_1$-dimensional  subspace of $\mathbb{C}^{m-2}$ in $M(m)$.
	
	Consider the polynomial ring $R=\mathbb{C}[\mathcal{X}]$, where 
	$\mathcal{X}=\{x(a,b) \, :\, a,b\in [1,m-3]\}\cup\{y(a,b) \, :\, a,b\in [1,m-2]\}$
	is algebraically independent. 
	For each $P=(i_1,\dots,i_{2n})\in\mathcal{I}$, 
	we will define the corresponding matrix $N^{(k)}(P)$ as follows. 
	By convention, let $i_{-1}=i_0=0$ and $i_{2n+1}=m-3$.
	
	For each pair $(\nu,\mu) \in [0,n]^2$, let $S_{\nu,\mu}$ be the following $(i_{2\nu}-i_{2\nu-1})\times (i_{2\mu+1}-i_{2\mu})$ matrix
	$$
	\left[ {\begin{array}{cccc}
			x(i_{2\nu-1}+1,i_{2\mu}+1) & x(i_{2\nu-1}+1,i_{2\mu}+2) & \cdots & x(i_{2\nu-1}+1,i_{2\mu+1})\\
			x(i_{2\nu-1}+2,i_{2\mu}+1) & x(i_{2\nu-1}+2,i_{2\mu}+2) & \cdots & x(i_{2\nu-1}+2,i_{2\mu+1})\\
			\vdots & \vdots & \ddots &\vdots\\
			x(i_{2\nu},i_{2\mu}+1) & x(i_{2\nu},i_{2\mu}+2) & \cdots & x(i_{2\nu},i_{2\mu+1})\\
	\end{array} } \right].$$
	Let $\widehat{S}_{\nu,\mu}$ be the $(i_{2\nu}-i_{2\nu-1})\times (i_{2\mu+1}-i_{2\mu-1})$ matrix 
	defined by 
	$$
	\left\{
	\begin{array}{ll}
		\begin{pmatrix}\mathbf{0}|S_{\nu,\mu}\end{pmatrix},& \text{ if }\nu<\mu;\\
		\begin{pmatrix}\mathbf{I}_{i_{2\nu}-i_{2\nu-1}}|S_{\nu,\mu}\end{pmatrix},
		& \text{ if } \nu=\mu;\\
		\text{the zero matrix},& \text{ if }\nu>\mu.\\
	\end{array}
	\right.
	$$
	Here $\mathbf{0}$ is the $(i_{2\nu}-i_{2\nu-1}) \times (i_{2\mu}-i_{2\mu-1})$ zero matrix.
	Let ${N}^{(2)}_{\nu,\mu}$ be the $(i_{2\nu+1}-i_{2\nu-1})\times (i_{2\mu+1}-i_{2\mu-1})$ matrix 
	defined by 
	$$
	\begin{pmatrix}\widehat{S}_{\nu,\mu}\\ \mathbf{0}\end{pmatrix}.
	$$
	Here $\mathbf{0}$ is the $(i_{2\nu+1}-i_{2\nu}) \times (i_{2\mu+1}-i_{2\mu-1})$ zero matrix. %Note that $N^{(2)}_{\nu,\mu}$ is an upper triangular matrix.

	\begin{definition}
		For each $P=(i_1,\dots,i_{2n})\in\mathcal{I}$, we define $N^{(2)}(P)$ by
		$$
		\begin{pmatrix}
			N^{(2)}_{0,0} & N^{(2)}_{0,1} & \cdots & N^{(2)}_{0,n}\\
			N^{(2)}_{1,0} & N^{(2)}_{1,1} & \cdots & N^{(2)}_{1,n}\\
			\vdots & \vdots & \ddots &\vdots\\
			N^{(2)}_{n,0} & N^{(2)}_{n,1} & \cdots & N^{(2)}_{n,n}\\
		\end{pmatrix}.
		$$
		
		Note that $N^{(2)}(P)$ is an $(m-3)\times(m-3)$ upper triangular matrix with entries in $R$.
		%and that $1\in R$ appears as a $(j,k)$-entry of $N^{(2)}(P)$ if and only if $i_{2\nu-1}+1 \leq j \leq i_{2\nu}$ for some $\nu$ and $j=k$.

		Let $\overline{N}^{(2)}(P)=N^{(2)}(P)(\mathbf{I}_{m-3} | \mathbf{0})$ and $\underline{N^{(2)}}(P)=N^{(2)}(P)(\mathbf{0}| \mathbf{I}_{m-3} )$, where $\mathbf{0}$ is the $(m-3)\times 1$ zero matrix. Let $\underline{\overline{N^{(2)}}}(P) = \begin{pmatrix}\overline{N^{(2)}}(P)\\ \underline{N^{(2)}}(P) \end{pmatrix}$. 
		Note that $\underline{\overline{N^{(2)}}}(P) $ is a $2(m-3) \times (m-2)$ matrix with entries in $R$.
		
		The row of $N^{(2)}(P)$ that contains $x(a,b)$ as an entry is called the row $a$.  
		The corresponding row in $\overline{N^{(2)}}(P)$ and the corresponding row in $\underline{N^{(2)}}(P)$
		will be written as $\overline{a}$ and $\underline{a}$, respectively. The corresponding rows in $\underline{\overline{N^{(2)}}}(P) $ will be still denoted by $\overline{a}$ and $\underline{a}$, respectively. 
		
		The column of $\overline{N^{(2)}}(P) $ that contains $x(a,b)$ will be called the column $b$, and the column of $\underline{N^{(2)}}(P) $ that contains $x(a,b-1)$ will be called the column $b$. The corresponding column in $\underline{\overline{N^{(2)}}}(P) $ will be still denoted by $b$.

	\end{definition}
	
	\begin{example}
		Let $m=11$. %and $j=2$. 
		% P=[2121212]1[2121212]12
		Let $P=(0,3,4,7)\in \mathcal{I}$, 
		%Let $P=(0,3,4,7)\in \mathcal{I}$, 
		that is, $i_1=0$, $i_2=3$, $i_3=4$ and $i_4=7$. 
		%Alternatively, $b_1=3$, $g_1=1$, $b_2=3$ and $g_2=1$.
		%By convention, $i_{-1}=i_0=0$ and $i_{2j+1}=m-3$.
		%Consider a perfect matching $[2121212]1[2121212]12$.
		
		For simplicity, for $a,b\in [1,m-3]$, let $x_{a,b}=x(a,b)\in R$.
		Note that $N^{(2)}_{\nu,\mu}$ is empty if $\min(\nu,\mu)=0$, and
		$$N^{(2)}_{1,1}=
		\begin{pmatrix}
			1 & 0 & 0 & x_{1,4}\\
			0 & 1 & 0 & x_{2,4}\\
			0 & 0 & 1 & x_{3,4}\\
			0 & 0 & 0& 0
		\end{pmatrix},
		\quad
		N^{(2)}_{1,2}=
		\begin{pmatrix}
			0 & 0 & 0 & x_{1,8}\\
			0 & 0 & 0 & x_{2,8}\\
			0 & 0 & 0 & x_{3,8}\\
			0 & 0 & 0& 0
		\end{pmatrix},
		$$
		$$N^{(2)}_{2,1}=
		\begin{pmatrix}
			0 & 0 & 0 &       0\\
			0 & 0 & 0 &       0\\
			0 & 0 & 0 &       0\\
			0 & 0 & 0& 0
		\end{pmatrix},
		\quad
		N^{(2)}_{2,2}=
		\begin{pmatrix}
			1 & 0 & 0 & x_{5,8}\\
			0 & 1 & 0 & x_{6,8}\\
			0 & 0 & 1 & x_{7,8}\\
			0 & 0 & 0& 0
		\end{pmatrix}.
		$$
		So
		$$N^{(2)}(P)=
		\begin{pmatrix}
			1 & 0 & 0 & x_{1,4} & 0 & 0& 0 & x_{1,8} \\
			0 & 1 & 0 & x_{2,4} & 0 & 0& 0 & x_{2,8} \\
			0 & 0 & 1 & x_{3,4} & 0 & 0& 0 & x_{3,8} \\
			0 & 0 & 0& 0&0 & 0 & 0& 0\\
			0 & 0 & 0 & 0 & 1 & 0& 0 & x_{5,8} \\
			0 & 0 & 0 & 0 & 0 & 1& 0 & x_{6,8} \\
			0 & 0 & 0 & 0 & 0 & 0& 1 & x_{7,8} \\
			0 & 0 & 0& 0&0 & 0 & 0& 0
		\end{pmatrix}.$$
		Then we get
		$$\underline{\overline{N^{(2)}}}(P) =
		\begin{pmatrix}
			1 & 0 & 0 & x_{1,4} & 0 & 0& 0 & x_{1,8} &0\\
			0 & 1 & 0 & x_{2,4} & 0 & 0& 0 & x_{2,8} &0\\
			0 & 0 & 1 & x_{3,4} & 0 & 0& 0 & x_{3,8} &0\\
			0 & 0 & 0& 0&0 & 0 & 0& 0&0\\
			0 & 0 & 0 & 0 & 1 & 0& 0 & x_{5,8} &0\\
			0 & 0 & 0 & 0 & 0 & 1& 0 & x_{6,8} &0\\
			0 & 0 & 0 & 0 & 0 & 0& 1 & x_{7,8}&0 \\
			0 & 0 & 0& 0&0 & 0 & 0& 0&0\\
			0&1 & 0 & 0 & x_{1,4} & 0 & 0& 0 & x_{1,8} \\
			0&0 & 1 & 0 & x_{2,4} & 0 & 0& 0 & x_{2,8} \\
			0&0 & 0 & 1 & x_{3,4} & 0 & 0& 0 & x_{3,8} \\
			0 & 0 & 0& 0&0 & 0 & 0& 0&0\\
			0&0 & 0 & 0 & 0 & 1 & 0& 0 & x_{5,8} \\
			0&0 & 0 & 0 & 0 & 0 & 1& 0 & x_{6,8} \\
			0&0 & 0 & 0 & 0 & 0 & 0& 1 & x_{7,8} \\
			0 & 0 & 0& 0&0 & 0 & 0& 0&0\\
		\end{pmatrix}
		$$
		The row $
		\begin{pmatrix}
			1 & 0 & 0 & x_{1,4} & 0 & 0& 0 & x_{1,8} &0\\
		\end{pmatrix}
		$ will be labeled by $\overline{1}$, the row $
		\begin{pmatrix}
			0 & 1 & 0 & x_{2,4} & \cdots\\
		\end{pmatrix}
		$  will be labeled by $\overline{2}$, and so on.
		Similarly, the row 
		$\begin{pmatrix}
			0&1 & 0 & 0 & x_{1,4} & 0 & 0& 0 & x_{1,8} \\
		\end{pmatrix}$
		will be labeled by $\underline{1}$, and so on.
		
	\end{example}
	
	\begin{example}
		Let $m=31$ and $P=(2,5,9,15,18,22)$. Then $N^{(2)}(P)$ looks like the following.
		$${\tiny\begin{array}{cccccccccccccccccccccccccccc}
				0&0&0&0&0&0&0&0&0&0&0&0&0&0&0&0&0&0&0&0&0&0&0&0&0&0&0&0\\
				0&0&0&0&0&0&0&0&0&0&0&0&0&0&0&0&0&0&0&0&0&0&0&0&0&0&0&0\\
				0&0&1&0&0&*&*&*&*&0&0&0&0&0&0&*&*&*&0&0&0&0&*&*&*&*&*&*\\
				0&0&0&1&0&*&*&*&*&0&0&0&0&0&0&*&*&*&0&0&0&0&*&*&*&*&*&*\\
				0&0&0&0&1&*&*&*&*&0&0&0&0&0&0&*&*&*&0&0&0&0&*&*&*&*&*&*\\
				0&0&0&0&0&0&0&0&0&0&0&0&0&0&0&0&0&0&0&0&0&0&0&0&0&0&0&0\\
				0&0&0&0&0&0&0&0&0&0&0&0&0&0&0&0&0&0&0&0&0&0&0&0&0&0&0&0\\
				0&0&0&0&0&0&0&0&0&0&0&0&0&0&0&0&0&0&0&0&0&0&0&0&0&0&0&0\\
				0&0&0&0&0&0&0&0&0&0&0&0&0&0&0&0&0&0&0&0&0&0&0&0&0&0&0&0\\
				0&0&0&0&0&0&0&0&0&1&0&0&0&0&0&*&*&*&0&0&0&0&*&*&*&*&*&*\\
				0&0&0&0&0&0&0&0&0&0&1&0&0&0&0&*&*&*&0&0&0&0&*&*&*&*&*&*\\
				0&0&0&0&0&0&0&0&0&0&0&1&0&0&0&*&*&*&0&0&0&0&*&*&*&*&*&*\\
				0&0&0&0&0&0&0&0&0&0&0&0&1&0&0&*&*&*&0&0&0&0&*&*&*&*&*&*\\
				0&0&0&0&0&0&0&0&0&0&0&0&0&1&0&*&*&*&0&0&0&0&*&*&*&*&*&*\\
				0&0&0&0&0&0&0&0&0&0&0&0&0&0&1&*&*&*&0&0&0&0&*&*&*&*&*&*\\
				0&0&0&0&0&0&0&0&0&0&0&0&0&0&0&0&0&0&0&0&0&0&0&0&0&0&0&0\\
				0&0&0&0&0&0&0&0&0&0&0&0&0&0&0&0&0&0&0&0&0&0&0&0&0&0&0&0\\
				0&0&0&0&0&0&0&0&0&0&0&0&0&0&0&0&0&0&0&0&0&0&0&0&0&0&0&0\\
				0&0&0&0&0&0&0&0&0&0&0&0&0&0&0&0&0&0&1&0&0&0&*&*&*&*&*&*\\
				0&0&0&0&0&0&0&0&0&0&0&0&0&0&0&0&0&0&0&1&0&0&*&*&*&*&*&*\\
				0&0&0&0&0&0&0&0&0&0&0&0&0&0&0&0&0&0&0&0&1&0&*&*&*&*&*&*\\
				0&0&0&0&0&0&0&0&0&0&0&0&0&0&0&0&0&0&0&0&0&1&*&*&*&*&*&*\\
				0&0&0&0&0&0&0&0&0&0&0&0&0&0&0&0&0&0&0&0&0&0&0&0&0&0&0&0\\
				0&0&0&0&0&0&0&0&0&0&0&0&0&0&0&0&0&0&0&0&0&0&0&0&0&0&0&0\\
				0&0&0&0&0&0&0&0&0&0&0&0&0&0&0&0&0&0&0&0&0&0&0&0&0&0&0&0\\
				0&0&0&0&0&0&0&0&0&0&0&0&0&0&0&0&0&0&0&0&0&0&0&0&0&0&0&0\\
				0&0&0&0&0&0&0&0&0&0&0&0&0&0&0&0&0&0&0&0&0&0&0&0&0&0&0&0\\
				0&0&0&0&0&0&0&0&0&0&0&0&0&0&0&0&0&0&0&0&0&0&0&0&0&0&0&0\\
		\end{array}}$$
		Here if the entry in the $a$-th row and the $b$-th column is marked with $*$, then it is equal to $x_{a,b}$.
	\end{example}
	
	%Now, if the number of rows in $\underline{\overline{N^{(2)}}}(P)$, which should be $2e_2(P)$ is less than $(e_1(P)+1)$, then $\mathcal{J}(P)$ is empty and
	%$$\sum_{j=1}^n (i_{2j}-i_{2j-1})<n.$$
	
	Next we want to define $N^{(1)}(P)$. 
	Define $\mathcal{A}(P)=\{i_{2j}: j \in [1,n] \text{ and } i_{2j}-i_{2j-1}=0\}$ and 
	$\mathcal{B}(P)=\{\ell \in [i_{2j-1}+1,i_{2j}]:j \in [1,n]\}.$

	Let $\{\mathbf{e}_k: 1 \leq k \leq m-2\}$ be the standard basis for $\mathbb{C}^{m-2}$. 
	Let $S'(P)$ be the $|\mathcal{A}(P)|\times (m-2)$ matrix where the $j$-th row is equal to
	$$\left\{\begin{array}{ll}
		\mathbf{0}, & \text{ if }i_{2j} \not\in \mathcal{A}(P);\\
		\mathbf{e}_{i_{2j}+1}+\sum\limits_{
			\substack{\ell=i_{2j}+2\\ \ell \notin \mathcal{A}(P)\cup \mathcal{B}(P)}}^{m-2} 
		y(i_{2j}+1,\ell)\mathbf{e}_\ell, & \text{ if }i_{2j}\in \mathcal{A}(P).\end{array}\right.$$
	
	Define
	$N^{(1)}(P) = \begin{pmatrix}\underline{\overline{N^{(2)}}}(P)\\ S'(P) \end{pmatrix}$.
	
	\begin{definition}
		For each $i\in\{1,2\}$, let $N_i(P)$ be the matrix obtained from $N^{(i)}(P)$ by removing zero rows.
	\end{definition}
	
	\begin{remark}
		When the variables are replaced with complex numbers, $N_2(P)$ will become a reduced row Echelon form.
	\end{remark}
	
	The row of $N_1(P)$ that contains $y(a,b)$ as an entry will be labeled by $a'$.  
	The column of $N_1(P) $ that contains $y(a,b)$ will be denoted by $b$.
	
	\begin{example}\label{example20240124}
		Let $P=(0,2,3,3,4,4,5,6)$ and $m-3=8$.
		Then     $$N_1(P)=
		\left(\begin{array}{ccccccccc}
			1 & 0 & x_{1,3} & x_{1,4} & x_{1,5} & 0 & x_{1,7} & x_{1,8} & 0\\
			0 & 1 & x_{2,3} & x_{2,4} & x_{2,5} & 0 & x_{2,7} & x_{2,8} & 0\\
			0 & 0 & 0       & 0       & 0       & 1 & x_{6,7} & x_{6,8} & 0\\
			0 & 1 & 0       & x_{1,3} & x_{1,4} & x_{1,5} & 0 & x_{1,7} & x_{1,8}\\
			0 & 0 & 1       & x_{2,3} & x_{2,4} & x_{2,5} & 0 & x_{2,7} & x_{2,8}\\
			0 & 0 & 0       & 0       & 0       & 0       & 1 & x_{6,7} & x_{6,8}\\
			0 & 0 & 0       & 1       & 0       & 0       & 0 & y_{4,8} & y_{4,9}\\
			0 & 0 & 0       & 0       & 1       & 0       & 0 & y_{5,8} & y_{5,9}\\
		\end{array}\right).$$
	\end{example}
	
	\section{The Fibonacci trees}
	We still fix $m\in \mathbb{Z}_{\ge 3}$ and $P=(i_1,...,i_{2n})\in\mathcal{I}_n$. 
	In Section~\ref{N2PN1P}, we have defined $N_2(P)$ and $N_1(P)$. Later we will substitute a certain complex number for each $x_{a,b}$ and $y_{a,b}$ in $N_2(P)$ and $N_1(P)$ to get a subrepresentation of $M(m)$. This substitution should be made in such a way that it gives an $e_2$-dimensional  subspace of $\mathbb{C}^{m-3}$ and an $e_1$-dimensional  subspace of $\mathbb{C}^{m-2}$ in $M(m)$. Note that the rows of $N_2(P)$ are in Echelon form, so its row space is of dimension $e_2$ as desired. On the other hand, the number of rows of $N_1(P)$ can be greater than $e_1$ in general. This means that there should be nontrivial relations between the complex numbers we substitute. In this section, we will find these relations. More precisely, we will give an explicit description for $\mathcal{J}(P)$, where $\mathcal{J}(P)$ is the ideal of $R$ generated 
	by the $(e_1(P)+1)\times (e_1(P)+1)$ minors of $N_1(P)$. 
	
	%Remember that a subrepresentation $N$ of $M(m)$ is a collection of subspaces $N_i \subseteq \mathbb{C}^{m-1-i}$ such that $\phi_j(N_2) \subseteq N_1$ for all $i,j \in [1,2]$.
	
	% Then the span of rows of $M(P)$ will become $N_2$, and the span of rows of $\overline{\underline{N^{(2)}}}(P)$  will become $N_1$. 

	%Let $$\overline{JK}=\{(j,k)\in\mathbb{Z}^2 \, : \, i_{2\nu-1} < j \leq i_{2\nu}\text{ and } i_{2\mu} < k \leq i_{2\mu+1} \text{ for some }\nu\text{ and }\mu\text{ with }1 \leq \nu \leq \mu \leq n\}.$$

	%\subsection{The $U$-tree}
	For each $\beta\in[0,2n-1]$ and for each pair $(\nu,\mu) \in [1, n]^2$ with $\nu\le \mu$, define $F_1^{(\nu,\mu)}(\beta)$ 
	to be the collection of elements $(m_0,m_1,\dots,m_{\beta})\in [\nu,\mu]^{\beta+1}$ such that
	$$\left\{\begin{array}{ll}
		\nu \leq m_{\beta} < m_{\beta-1} \leq \cdots \leq m_2 < m_1 \leq m_0 =\mu & \text{ if }\beta\text{ is even};\\
		\nu \leq m_{\beta} \leq m_{\beta-1} < \cdots \leq m_2 < m_1 \leq m_0 =\mu &  \text{ if }\beta\text{ is odd,} \end{array}\right.$$ 
	and define $F_2^{(\nu,\mu)}(\beta-1)$ to be the collection of elements 
	$(m_0,m_1,...,m_{\beta-1})\in [\nu,\mu]^{\beta}$ such that
	$$\left\{\begin{array}{ll}
		\nu \leq m_{\beta-1} < \cdots < m_2 \leq m_1 < m_0 =\mu & \text{ if }\beta\text{ is even};\\
		\nu \leq m_{\beta-1} \leq \cdots < m_2 \leq m_1 < m_0 =\mu &  \text{ if }\beta\text{ is odd.} \end{array}\right.$$ 
	
	\begin{definition}
		Let $$F_\eta^{(\nu,\mu)}=\bigcup\limits_{\beta\in[0,2n-1]} F_\eta^{(\nu,\mu)}(\beta)$$ 
		where $1 \leq \eta \leq 2$.
		If $(m_0,m_1,\dots,m_{\beta})\in F_\eta^{(\nu,\mu)}$, then we say that   $(m_0,m_1,\dots,m_{\beta})$ is a \emph{child} of $(m_0,m_1,\dots,m_{\beta-1})$.
		We define the graph $F_\eta^{(\nu,\mu)}(P)$ as the rooted tree such that the vertices are elements of
		$F_\eta^{(\nu,\mu)}$ and there will be an edge between vertices if and only if one is a child of the other. 
		This tree is called a \emph{Fibonacci tree}.
	\end{definition}
	
	\begin{example}\label{example20240110}
		Assume $n=3$. Then
		$$\aligned
		F_1^{(1,3)}=&\{(3),(3,3),(3,2),(3,1),(3,3,2),(3,3,1),(3,2,1), \\
		& (3,3,2,2),(3,3,2,1), (3,3,1,1),(3,2,1,1),(3,3,2,2,1),(3,3,2,2,1,1)\}.
		\endaligned$$
		and
		$$\aligned
		F_2^{(1,3)}=&\{(3),(3,2),(3,1),(3,2,2),(3,2,1), (3,1,1)\\
		& (3,2,2,1),(3,2,2,1,1)\}.
		\endaligned$$
		The associated trees $F_1^{(1,3)}(P)$ and $F_2^{(1,3)}(P)$ respectively, are
		
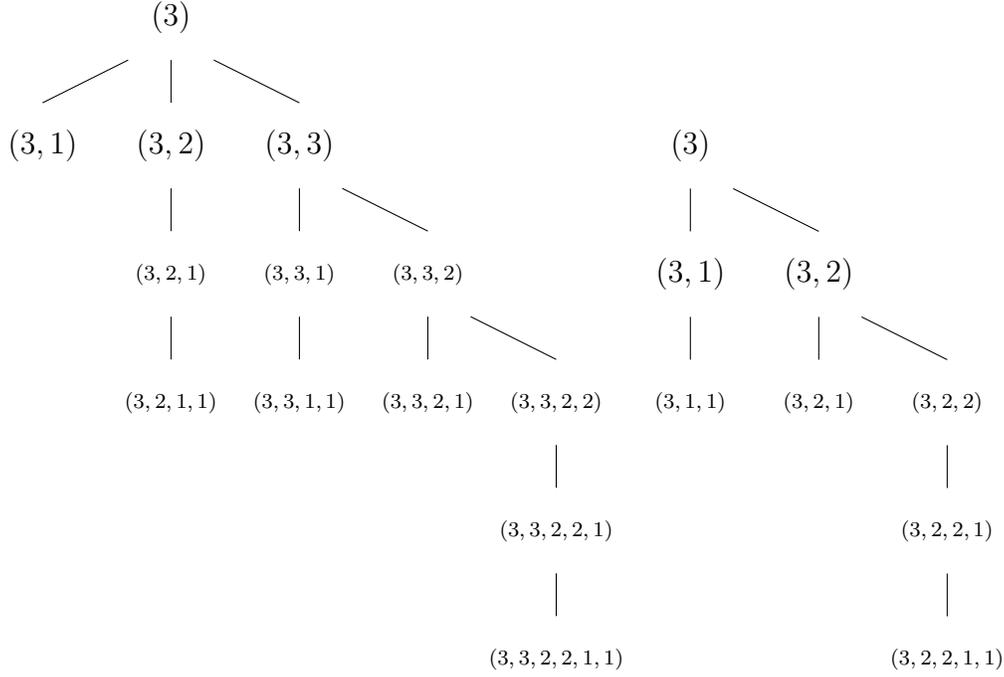
\begin{center}
			\begin{tikzpicture}[scale=0.4mm] 
				\node at (0,0.5) {$(3)$};
				\draw (-.5,0) -- (-1.5,-0.5);
				\draw (0,0) -- (0,-0.5);
				\draw (0.5,0) -- (1.5,-0.5);
				
				\node at (-1.5,-1) {$(3,1)$};
				
				\node at (0,-1) {$(3,2)$};
				\draw (0,-1.5) -- (0,-2);
				\node at (0,-2.5) {\tiny{$(3,2,1)$}};
				\draw (0,-3) -- (0,-3.5);
				\node at (0,-4) {\tiny{$(3,2,1,1)$}};
				
				\node at (1.5,-1) {$(3,3)$};
				\draw (1.5,-1.5) -- (1.5,-2);
				\draw (2,-1.5) -- (3,-2);
				\node at (1.5,-2.5) {\tiny{$(3,3,1)$}};
				\draw (1.5,-3) -- (1.5,-3.5);
				\node at (1.5,-4) {\tiny{$(3,3,1,1)$}};
				
				\node at (3,-2.5) {\tiny{$(3,3,2)$}};
				\draw (3,-3) -- (3,-3.5);
				\draw (3.5,-3) -- (4.5,-3.5);
				\node at (3,-4) {\tiny{$(3,3,2,1)$}};
				\node at (4.5,-4) {\tiny{$(3,3,2,2)$}};
				\draw (4.5,-4.5) -- (4.5,-5);
				\node at (4.5,-5.5) {\tiny{$(3,3,2,2,1)$}};
				\draw (4.5,-6) -- (4.5,-6.5);
				\node at (4.5,-7) {\tiny{$(3,3,2,2,1,1)$}};
			\end{tikzpicture}
			\begin{tikzpicture}[scale=0.4mm] 
				\node at (0,0.5) {$(3)$};
				\draw (0,0) -- (0,-0.5);
				\draw (0.5,0) -- (1.5,-0.5);
				
				\node at (0,-1) {$(3,1)$};
				\draw (0,-1.5) -- (0,-2);
				\node at (0,-2.5) {\tiny{$(3,1,1)$}};
				
				\node at (1.5,-1) {$(3,2)$};
				\draw (1.5,-1.5) -- (1.5,-2);
				\draw (2,-1.5) -- (3,-2);
				\node at (1.5,-2.5) {\tiny{$(3,2,1)$}};

				\node at (3,-2.5) {\tiny{$(3,2,2)$}};
				\draw (3,-3) -- (3,-3.5);
				\node at (3,-4) {\tiny{$(3,2,2,1)$}};
				\draw (3,-4.5) -- (3,-5);
				\node at (3,-5.5) {\tiny{$(3,2,2,1,1)$}};
			\end{tikzpicture}
		\end{center}
		
		The name 
		\emph{Fibonacci tree} is coined by the following:
		
		\begin{lemma}
			The number of vertices in each $F_\eta^{(\nu,\mu)}(P)$ is equal to a Fibonacci number.
		\end{lemma}
		\begin{proof}
			Straightforward.
		\end{proof}
		
		Note that, in Example~\ref{example20240110},  the tree $F_\eta^{(3,3)}(P)$ is a subtree of the tree $F_\eta^{(2,3)}(P)$, which is a subtree of $F_\eta^{(1,3)}(P)$.
	\end{example}
	
	The following lemma is clear.
	\begin{lemma}
		The trees $F_\eta^{(\nu,\mu)}(P)$ and $F_\eta^{(\nu',\mu')}(P)$ are isomorphic to each other 
		if and only if $\mu-\nu=\mu'-\nu'$.
	\end{lemma}
	
	\begin{definition}
		We define another tree $\widetilde{F}_\eta^{(\nu,\mu)}(P)$ to be obtained from $F_\eta^{(\nu,\mu)}(P)$ by adding to each vertex $(m_0,\dots,m_\beta)$ its child  $\widetilde{(m_0,\dots,m_\beta)}$. This tree will be called a \emph{framed Fibonacci tree}. Let $\widetilde{V}_\eta^{(\nu,\mu)}=\widetilde{V}_\eta^{(\nu,\mu)}(P)$ be the set of vertices of $\widetilde{F}_\eta^{(\nu,\mu)}(P)$.
	\end{definition}
	
	\begin{example}
		Consider ${F}_2^{(1,3)}(P)$ as in the previous example.
		Then $\widetilde{F}_2^{(1,3)}(P)$ is given below.
		
		$$
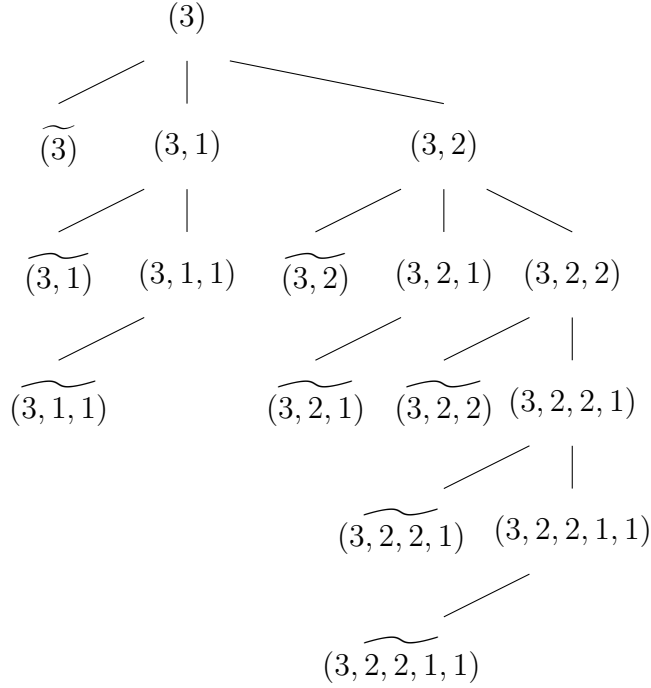
\begin{tikzpicture}[scale=0.4mm] 
			\node at (0,0.5) {$(3)$};
			\draw (-0.5,0) -- (-1.5,-0.5);
			\draw (0,0) -- (0,-0.5);
			\draw (0.5,0) -- (3,-0.5);
			
			\node at (-1.5,-1) {$\widetilde{(3)}$};
			
			\node at (0,-1) {$(3,1)$};
			\draw (-0.5,-1.5) -- (-1.5,-2);
			\draw (0,-1.5) -- (0,-2);
			\node at (-1.5,-2.5) {$\widetilde{(3,1)}$};
			\node at (0,-2.5) {$(3,1,1)$};
			\draw (-0.5,-3) -- (-1.5,-3.5);
			\node at (-1.5,-4) {$\widetilde{(3,1,1)}$};
			
			\node at (3,-1) {$(3,2)$};
			\draw (2.5,-1.5) -- (1.5,-2);
			\draw (3,-1.5) -- (3,-2);
			\draw (3.5,-1.5) -- (4.5,-2);
			\node at (1.5,-2.5) {$\widetilde{(3,2)}$};
			\node at (3,-2.5) {$(3,2,1)$};
			\draw (2.5,-3) -- (1.5,-3.5);
			\node at (1.5,-4) {$\widetilde{(3,2,1)}$};
			
			\node at (4.5,-2.5) {$(3,2,2)$};
			\draw (4,-3) -- (3,-3.5);
			\draw (4.5,-3) -- (4.5,-3.5);
			\node at (3,-4) {$\widetilde{(3,2,2)}$};
			
			\node at (4.5,-4) {$(3,2,2,1)$};
			\draw (4,-4.5) -- (3,-5);
			\draw (4.5,-4.5) -- (4.5,-5);
			
			\node at (2.5,-5.5) {$\widetilde{(3,2,2,1)}$};
			
			\node at (4.5,-5.5) {$(3,2,2,1,1)$};
			\draw (4,-6) -- (3,-6.5);
			\node at (2.5,-7) {$\widetilde{(3,2,2,1,1)}$};
		\end{tikzpicture}$$
	\end{example}
	
	\begin{definition}
		Consider $P=(i_1,\dots,i_{2n})\in \mathcal I$, and 
		fix a pair $(\nu,\mu)$ of integers with $1 \leq \nu \leq \mu \leq n$, 
		and another pair $(j,k)$ of integers such that $i_{2\nu-1}+1 \leq j \leq i_{2\nu}$ 
		and $i_{2\mu}+1 \leq k \leq i_{2\mu+1}$.
		
		The \textit{set of leading terms} for $(j,k)$, denoted $\mathcal{L}^P(j,k)$, is defined as follows. 
		$$\mathcal{L}^P(j,k)=
		\begin{cases}
			\{x(i_{2n-1}+1,i_{2n+1})\}, &\text{if $j=i_{2n-1}+1$, $k=i_{2n+1}$ and $i_{2n}+1 < i_{2n+1}$};\\
			\{x(j,i_{2n+1})\}, &\text{if $k=i_{2n+1}$ and $i_{2n}+1 = i_{2n+1}$};\\
			\{x(j,k),x(j+1,k+1)\}, &\text{if $i_{2\nu-1}+1 \leq j < i_{2\nu}$ and $i_{2\mu}+1 \leq k < i_{2\mu+1}$};\\
			\emptyset, &\text{otherwise}.
		\end{cases}
		$$
		
		Note that the condition $i_{2n}+1 < i_{2n+1}$ is saying that the ``last group of variables" consists of more than one column.
		\begin{example}
			Consider $P=(0,2,3,5)$ where $m-3=7$. Then
			$$N_2(P)=
			\left(\begin{array}{ccccccccc}
				1 & 0 & x_{1,3} & 0 & 0 & x_{1,6} & x_{1,7}\\
				0 & 1 & x_{2,3} & 0 & 0 & x_{2,6} & x_{2,7}\\
				0 & 0 & 0       & 1 & 0 & x_{4,6} & x_{4,7}\\
				0 & 0 & 0       & 0 & 1 & x_{5,6} & x_{5,7}\\
			\end{array}\right)$$ and the possible nonempty sets of leading terms are
			$\mathcal{L}^P(1,7)=\{x_{1,7}\}$, $\mathcal{L}^P(4,7)=\{x_{4,7}\}$, $\mathcal{L}^P(1,6)=\{x_{1,6},x_{2,7}\}$
			and $\mathcal{L}^P(4,6)=\{x_{4,6},x_{5,7}\}$.
			
			If $m-3=6$ and $P=(0,2,3,5)$, then we have
			$$N_2(P)=
			\left(\begin{array}{ccccccccc}
				1 & 0 & x_{1,3} & 0 & 0 & x_{1,6}\\
				0 & 1 & x_{2,3} & 0 & 0 & x_{2,6}\\
				0 & 0 & 0       & 1 & 0 & x_{4,6}\\
				0 & 0 & 0       & 0 & 1 & x_{5,6}\\
			\end{array}\right)$$ and the possible nonempty sets of leading terms are
			$\mathcal{L}^P(1,6)=\{x_{1,6}\}$, $\mathcal{L}^P(2,6)=\{x_{2,6}\}$, $\mathcal{L}^P(4,6)=\{x_{4,6}\}$
			and $\mathcal{L}^P(5,6)=\{x_{5,6}\}$.
		\end{example}
		
		Note that  $|\mathcal{L}^P(j,k)|\in\{0,1,2\}$. Assume that $|\mathcal{L}^P(j,k)|\ge 1$. 
		Then, for each $\eta\in[1,|\mathcal{L}^P(j,k)|]$, we define a function $\varphi_\eta^{(j,k)} : \widetilde{V}^{(\nu,\mu)}\longrightarrow R$ as follows. 
		$$\aligned 
		&\varphi_\eta^{(j,k)}(m_0)=1\\
		&\varphi_\eta^{(j,k)}(\widetilde{(m_0)})=\left\{\begin{array}{ll}
			x(j,k), &\text{ if }\eta=1;\\
			x(j+1,k+1), &\text{ if }\eta=2;\\
		\end{array}\right.\\
		&\varphi_\eta^{(j,k)}(m_0,m_1)=\left\{\begin{array}{ll}
			x(i_{2m_1},k), &\text{ if }\eta=1;\\
			x(i_{2m_1+1}+1,k+1), &\text{ if }\eta=2;\\
		\end{array}\right.\\
		&\varphi_\eta^{(j,k)}(\widetilde{(m_0,m_1)})=\left\{\begin{array}{ll}
			x(j+1,i_{2m_1}+1), &\text{ if }\eta=1;\\
			x(j,i_{2m_1+1}), &\text{ if }\eta=2.\\
		\end{array}\right.\\
		\endaligned$$
		If $2 \leq \alpha \leq \mu-\nu$ is an even integer, then
		$$\aligned 
		&\varphi_\eta^{(j,k)}(m_0,\dots,m_\alpha)=\left\{\begin{array}{ll}
			x(i_{2m_\alpha+1}+1,i_{2m_{\alpha-1}}+1), &\text{ if }\eta=1;\\
			x(i_{2m_\alpha},i_{2m_{\alpha-1}+1}), &\text{ if }\eta=2;\\
		\end{array}\right.\\
		&\varphi_\eta^{(j,k)}(\widetilde{(m_0,\dots,m_\alpha)})=\left\{\begin{array}{ll}
			x(j,i_{2m_\alpha+1}), &\text{ if }\eta=1;\\
			x(j+1,i_{2m_\alpha}+1), &\text{ if }\eta=2.\\
		\end{array}\right.\\
		\endaligned$$
		If $3 \leq \alpha \leq \mu-\nu$ is an odd integer, then
		$$\aligned &\varphi_\eta^{(j,k)}(m_0,\dots,m_{\alpha})=\left\{\begin{array}{ll}
			x(i_{2m_{\alpha}},i_{2m_{\alpha-1}+1}), &\text{ if }\eta=1;\\
			x(i_{2m_{\alpha}}+1,i_{2m_{\alpha-1}}), &\text{ if }\eta=2;\\
		\end{array}\right.\\
		&\varphi_\eta^{(j,k)}(\widetilde{(m_0,\dots,m_\alpha)})=\left\{\begin{array}{ll}
			x(j+1,i_{2m_{\alpha}}+1), &\text{ if }\eta=1;\\
			x(j,i_{2m_{\alpha}+1}), &\text{ if }\eta=2.\\
		\end{array}\right.\\
		\endaligned$$
		We will say that $x\in R$ is a \textit{child} of $y\in R$ if there exist
		$\mathbf{m}_1, \mathbf{m}_2 \in \widetilde{V}^{(\nu,\mu)}$ such that
		$\varphi_\eta^{(j,k)}(\mathbf{m}_1)=x$, $\varphi_\eta^{(j,k)}(\mathbf{m}_2)=y$ and $\mathbf{m}_1$
		is a child of $\mathbf{m}_2$.
		When convenient, we will identify the vertices of $\widetilde{F}_\eta^{(\nu,\mu)}(P)$ with
		their images via $\varphi_\eta^{(j,k)}$.
	\end{definition}

	\begin{example}
		Consider $m-3=12$ and $P=(0,2,4,6,8,10)$. One of the leading terms is $x_{2,12} \in \mathcal{L}^P(1,11)$.

		The tree $\widetilde{F}_2^{(1,3)}(P)$ may be written as
		$${\tiny
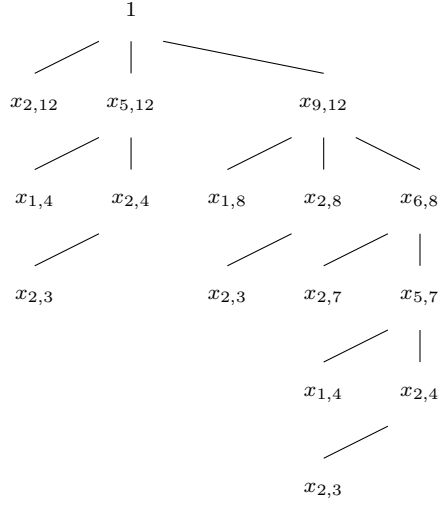
\begin{tikzpicture}[scale=0.3mm] 
				\node at (0,0.5) {$1$};
				\draw (-0.5,0) -- (-1.5,-0.5);
				\draw (0,0) -- (0,-0.5);
				\draw (0.5,0) -- (3,-0.5);
				
				\node at (-1.5,-1) {$x_{2,12}$};
				
				\node at (0,-1) {$x_{5,12}$};
				\draw (-0.5,-1.5) -- (-1.5,-2) ;
				\draw (0,-1.5) -- (0,-2);
				\node at (-1.5,-2.5) {$x_{1,4}$};
				\node at (0,-2.5) {$x_{2,4}$};
				\draw (-0.5,-3) -- (-1.5,-3.5);
				\node at (-1.5,-4) {$x_{2,3}$};
				
				\node at (3,-1) {$x_{9,12}$};
				\draw (2.5,-1.5) -- (1.5,-2);
				\draw (3,-1.5) -- (3,-2);
				\draw (3.5,-1.5) -- (4.5,-2);
				\node at (1.5,-2.5) {$x_{1,8}$};
				\node at (3,-2.5) {$x_{2,8}$};
				\draw (2.5,-3) -- (1.5,-3.5);
				\node at (1.5,-4) {$x_{2,3}$};
				
				\node at (4.5,-2.5) {$x_{6,8}$};
				\draw (4,-3) -- (3,-3.5);
				\draw (4.5,-3) -- (4.5,-3.5);
				\node at (3,-4) {$x_{2,7}$};
				
				\node at (4.5,-4) {$x_{5,7}$};
				\draw (4,-4.5) -- (3,-5);
				\draw (4.5,-4.5) -- (4.5,-5);
				
				\node at (3,-5.5) {$x_{1,4}$};
				
				\node at (4.5,-5.5) {$x_{2,4}$};
				\draw (4,-6) -- (3,-6.5);
				\node at (3,-7) {$x_{2,3}$};
		\end{tikzpicture}}$$
	\end{example}
	
	\begin{definition}\label{lset}
		Define 
		$$\aligned 
		&D^P_{j,k} =
		\sum_{\eta=1}^{|\mathcal L(j,k)|} (-1)^{\eta-1}\sum_{(m_0,\dots,m_\beta)\in F_\eta^{(\nu,\mu)}} \varphi_\eta^{(j,k)}(\widetilde{(m_0,\dots,m_\beta)})\prod_{\ell\in[0,\beta]} \varphi_\eta^{(j,k)}(m_0,\dots,m_\ell).
		\endaligned$$
		and define
		$$\widehat{D^P_{j,k}}=\left\{\begin{array}{ll}
			D^P_{j,k}, & \text{ if } \mathcal{A}(P)=\emptyset;\\ 
			&\\
			D^{P'}_{j,k}-\sum\limits_{i_{2\beta} \in \mathcal{A}(P)} y(i_{2\beta}+1,k+1)D^{P'}_{j,i_{2\beta}}, & \text{ if } \mathcal{A}(P)\neq\emptyset,\end{array}\right.$$
		where $D^{P'}_{j,i_{2\beta_\ell}}$ is defined as 
		$D^{P'}_{j,i_{2\beta}}$ with $P' \in \mathcal{I}_{n-|\mathcal{A}(P)|}$ obtained from $P$ by removing the entries $(i_{2\beta-1},i_{2\beta})$ for every $i_{2\beta} \in \mathcal{A}(P)$. 
		
		Note that $\widehat{D^P_{j,k}}=D^P_{j,k}$ if
		\begin{enumerate}
			\item $\mathcal{A}(P)=\emptyset$ or
			\item $\mathcal{L}^{P'}(j,i_{2\beta})=\emptyset$ for every $i_{2\beta} \in \mathcal{A}(P)$. 
		\end{enumerate}
		
		We also define $L^P_{j,k}$ as the linear part of $\widehat{D^P_{j,k}}$, that is,
		$$L^P_{j,k}=\left\{
		\begin{array}{ll}
			0, & \text{ if }|\mathcal{L}^P(j,k)|=0; \\
			x(j,k),   & \text{ if }|\mathcal{L}^P(j,k)|=1; \\
			x(j,k)- x(j+1,k+1),  & \text{ if }|\mathcal{L}^P(j,k)|=2.
		\end{array} \right.
		$$
	\end{definition}
	
	\begin{remark}\label{rem}
		It is easy to see that $\widehat{D^P_{j,k}}$ and $L^P_{j,k}$ have the following features.
		
		\noindent(1) For fixed $j$ and $k$, the variables $x_{j,k}$ 
		and/or $x_{j+1,k+1}$ appear only in the linear part of the equation $\widehat{D^P_{j,k}}$. 
		
		\noindent(2) Let
		$$JK:=\bigcup\limits_{1\le\nu\le\mu\le n}\{(j,k)\, : \, i_{2\nu-1}+1 \leq j \leq i_{2\nu},\ i_{2\mu}+1 \leq k \leq i_{2\mu+1}, \mathcal{L}^P(j,k)\neq\emptyset\text{ and } k \notin \mathcal{A}(P)
		\}.$$
		Then the set
		$D(P):=\{L^P_{j,k}\, : (j,k)\in JK\}$ is linearly independent.
		
		\noindent(3) If we define a relation on the set $JK$ such that $(j,k) \leq (j',k')$ if and only if 
		$x(j,k)$ appears in a nonlinear term of $\widehat{D^P_{j',k'}}$, then this relation becomes a partial order.
	\end{remark}
	
	\begin{example}
		Consider $P=(0,2,4,5)$ where $m-3=8$. Then
		$$N_2(P)=
		\left(\begin{array}{cccccccccc}
			1 & 0 & x_{1,3} & x_{1,4} & 0 & 0 & x_{1,7} & x_{1,8}\\
			0 & 1 & x_{2,3} & x_{2,4} & 0 & 0 & x_{2,7} & x_{2,8}\\
			0 & 0 & 0       & 0       & 1 & 0 & x_{5,7} & x_{5,8}\\
			0 & 0 & 0       & 0       & 0 & 1 & x_{6,7} & x_{6,8}\\
		\end{array}\right).$$
		Recall $\mathcal{L}^P(1,7)=\{x_{1,7},x_{2,8}\}$.
		In this case,
		$\widetilde{F}^{(1,7)}_1 \cup \widetilde{F}^{(1,7)}_2$ is \\
		$${\tiny\begin{tikzpicture}[scale=0.3mm] 
				\node at (0,0.5) {$1$};
				\draw (-0.5,0) -- (-1.5,-0.5);
				\draw (0,0) -- (0,-0.5);
				\draw (0.5,0) -- (1.5,-0.5);
				
				\node at (-1.5,-1) {$x_{1,7}$};
				
				\node at (0,-1) {$x_{2,7}$};
				\draw (-0.5,-1.5) -- (-1.5,-2);
				\node at (-1.5,-2.5) {$x_{2,3}$};
				
				\node at (1.5,-1) {$x_{6,7}$};
				\draw (1,-1.5) -- (0,-2);
				\node at (0,-2.5) {$x_{2,7}$};
				\draw (1.5,-1.5) -- (1.5,-2);
				
				\node at (1.5,-2.5) {$x_{5,7}$};    
				\draw (1,-3) -- (0,-3.5);
				\node at (0,-4) {$x_{1,5}$};
				\draw (1.5,-3) -- (1.5,-3.5);
				\node at (1.5,-4) {$x_{2,4}$};
				\draw (1,-4.5) -- (0,-5);
				\node at (0,-5.5) {$x_{1,5}$};
		\end{tikzpicture}} \hspace{50pt}
		{\tiny
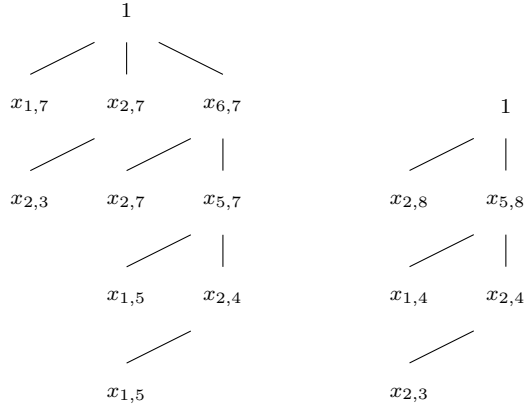
\begin{tikzpicture}[scale=0.3mm] 
				\node at (0,0.5) {$1$};
				\draw (-0.5,0) -- (-1.5,-0.5);
				\draw (0,0) -- (0,-0.5);
				
				\node at (-1.5,-1) {$x_{2,8}$};
				
				\node at (0,-1) {$x_{5,8}$};
				\draw (-0.5,-1.5) -- (-1.5,-2) ;
				\draw (0,-1.5) -- (0,-2);
				\node at (-1.5,-2.5) {$x_{1,4}$};
				\node at (0,-2.5) {$x_{2,4}$};
				\draw (-0.5,-3) -- (-1.5,-3.5);
				\node at (-1.5,-4) {$x_{2,3}$}; 
		\end{tikzpicture}}$$
		To compute $\widehat{D^P_{1,7}},$ each term will be the product of vertices of the unique path starting at a leaf of the forest $\widetilde{F}^{(1,7)}_1 \cup \widetilde{F}^{(1,7)}_2$ to the vertex 1 contained in the same path component.
		If the path is completely contained in $\widetilde{F}^{(1,7)}_1$, 
		the coefficient of the corresponding term will be $+1$.
		If the path is completely contained in $\widetilde{F}^{(1,7)}_2$, 
		the coefficient of the corresponding term will be $-1$.
		Then,
		\begin{align*}
			\widehat{D^P_{1,7}}=D^P_{1,7}=&(x_{1,7}+x_{2,7}x_{2,3}+x_{6,7}x_{2,7}+x_{6,7}x_{5,7}x_{1,4}+x_{6,7}x_{5,7}x_{2,4}x_{2,3})\\
			&-(x_{2,8}+x_{5,8}x_{1,4}+x_{5,8}x_{2,4}x_{2,3}).
		\end{align*}
		Furthermore,
		\begin{enumerate}
			\item $JK=\{(1,3),(1,7),(1,8),(5,7)\}$ and 
			\item $D(P)=\{x_{1,3}-x_{2,4},\ x_{1,7}-x_{2,8},\ x_{1,8},\ x_{5,7}-x_{6,8}\}$.
		\end{enumerate}
	\end{example}

	\begin{example}
		Consider the $P=(0,2,3,3,4,4,5,6)$ where $m-3=8$.
		Then $N_1(P)$ is given in Example~\ref{example20240124}. We get
		\begin{align*}    
			\widehat{D^P_{1,8}}=&x_{1,8}+x_{6,8}x_{2,7}+x_{2,8}x_{2,3}+x_{6,8}x_{6,7}x_{1,5}+x_{6,8}x_{6,7}x_{2,5}x_{2,3}\\
			&-y_{4,9}(x_{2,3}^2+x_{1,3}-x_{2,4})-y_{5,9}(x_{2,3}x_{2,4}+x_{1,4}-x_{2,5})\\
			&=D^{P}_{1,8}-y_{4,9}D^{P'}_{1,3}-y_{5,9}D^{P'}_{1,4}\text{ and}\\    
			\widehat{D^P_{1,7}}=&x_{1,7}-x_{2,8}-x_{6,8}x_{1,5}+x_{6,7}x_{2,7}+x_{2,7}x_{2,3}-x_{6,8}x_{2,5}x_{2,3}+x_{6,7}^2x_{1,5}+x_{6,7}^2x_{2,5}x_{2,3}\\
			&-y_{4,8}(x_{2,3}^2+x_{1,3}-x_{2,4})-y_{5,8}(x_{2,3}x_{2,4}+x_{1,4}-x_{2,5})\\
			&=D^{P}_{1,7}-y_{4,8}D^{P'}_{1,3}-y_{5,8}D^{P'}_{1,4}
		\end{align*}
		where $P'=(0,2,5,6)$.
	\end{example}

	\section{The main theorem}
	In this section, we present our main theorem, which can be used to explicitly construct every subrepresentation of $M(m)$. 
	
	Let $\mathcal{K}(P)$ be the ideal of $R$  generated by
	$\{\widehat{D^P_{j,k}}: (j,k)\in JK\}$. 
	Recall that $\mathcal{J}(P)$ is the ideal of $R$ generated 
	by the $(e_1(P)+1)\times (e_1(P)+1)$ minors of $N_1(P)$.
	The following is our main theorem.
	
	\begin{theorem}
		We have $\mathcal{J}(P)=\mathcal{K}(P)$. 
	\end{theorem}
	
	\begin{remark}\label{remark20240119}
		As a corollary, we can construct every subrepresentation of $M(m)$. Let $W(P)$ be the set of all variables appearing in $N_1(P)$. Thanks to Remark~\ref{rem}, there are two subsets $A$ and $B$ of $W(P)$ such that $|A|=|D(P)|$, $A\cup B=W(P)$, $A\cap B=\emptyset$, and that we can substitute an arbitrary complex number for each variable in $B$, which will uniquely determine a complex number for  each variable in $A$ satisfying all $\widehat{D^P_{j,k}}=0$. 
		This uniquely determines a subrepresentation of dimension $(e_1,e_2)$.
		By letting complex  numbers vary for variables in $B$, we get all subrepresentations of dimension $(e_1,e_2)$ in the cell corresponding to $P$.
	\end{remark}
	
	\begin{example}
		One possible choice for $A$ is described as follows.
		Let $P=(i_1,i_2,\dots,i_{2n}) \in \mathcal{I}$.
		Then 
		
		$$
		A=\{x(j,k) : \mathcal{L}^P(j,k) \neq \emptyset\}.
		$$

		For example, consider the case where $P=(0,2,4,4,5,6)$ and $m-3=8$.
		Then
		$$N_2(P)=
		\left(\begin{array}{cccccccccc}
			1 & 0 & x_{1,3} & x_{1,4} & x_{1,5} & 0 & x_{1,7} & x_{1,8}\\
			0 & 1 & x_{2,3} & x_{2,4} & x_{2,5} & 0 & x_{2,7} & x_{2,8}\\
			0 & 0 & 0       & 0       & 0       & 1 & x_{6,7} & x_{6,8}\\
		\end{array}\right).$$
		and
		$$S'(P)=
		\left(\begin{array}{cccccccccc}
			0 & 0 & 0 & 0 & 1 & 0 & 0 & y_{5,8} & y_{5,9}
		\end{array}\right).
		$$
		In this case, $D(P)=\{x_{1,3}-x_{2,4},\ x_{1,7}-x_{2,8},\ x_{1,8}\}.$
		Thus, $A=\{x_{1,3},\ x_{1,7},\ x_{1,8}\}$
		and $B=\{x_{1,4},\ x_{1,5},\ x_{2,3},\ x_{2,4},\ x_{2,5},\ x_{2,7},\ x_{2,8},\ x_{6,7},\ x_{6,8},\ y_{5,8},\ y_{5,9}\}$ 
		since
		
		$$\aligned
		\widehat{D^P_{1,3}}=&x_{1,3}-x_{2,4}+x_{2,3}^2\\
		=&D^{P'}_{1,3},\\
		\widehat{D^P_{1,7}}=&x_{1,7}-x_{2,8}+x_{2,7}x_{6,7}-x_{1,5}x_{6,8}+x_{2,3}x_{2,7}-x_{2,3}x_{2,5}x_{6,8}+x_{1,5}x_{6,7}^2+
		x_{2,3}x_{2,5}x_{6,7}^2\\
		&-y_{5,8}(x_{1,4}-x_{2,5}+x_{2,3}x_{2,4})\\
		=&D^P_{1,7}-y_{5,8}D^{P'}_{1,4},\\
		\widehat{D^P_{1,8}}=&x_{1,8}+x_{2,7}x_{6,8}+x_{2,3}x_{2,8}+x_{1,5}x_{6,7}x_{6,8}+x_{2,3}x_{2,5}x_{6,7}x_{6,8}\\
		&-y_{5,9}(x_{1,4}-x_{2,5}+x_{2,3}x_{2,4})\\
		=&D^P_{1,8}-y_{5,9}D^{P'}_{1,4}
		\endaligned$$
		where $P'=(0,2,6,7)$.
		
		Now, let every variable of $B$ be equal to 1. Then
		\begin{enumerate}
			\item $\widehat{D^P_{1,3}}=0$ implies $x_{1,3}=0$,
			\item $\widehat{D^P_{1,7}}=0$ implies $x_{1,7}=0$, and 
			\item $\widehat{D^P_{1,8}}=0$ implies $x_{1,8}=-3$.
		\end{enumerate}
		Thus the corresponding subrepresentation is $(N_2,N_1)$ where
		$N_2$ is the row space of 
		$${\tiny\left(\begin{array}{cccccccccc}
				1 & 0 & 0 & 1 & 1 & 0 & 0 & -3\\
				0 & 1 & 1 & 1 & 1 & 0 & 1 & 1\\
				0 & 0 & 0 & 0 & 0 & 1 & 1 & 1\\
			\end{array}\right)}$$
		and $N_1$ is the row space of 
		$${\tiny\left(\begin{array}{cccccccccc}
				1 & 0 & 0 & 1 & 1 & 0 & 0  & -3 & 0\\
				0 & 1 & 1 & 1 & 1 & 0 & 1  & 1  & 0\\
				0 & 0 & 0 & 0 & 0 & 1 & 1  & 1  & 0\\
				0 & 1 & 0 & 0 & 1 & 1 & 0  & 0  & -3\\
				0 & 0 & 1 & 1 & 1 & 1 & 0  & 1  & 1\\
				0 & 0 & 0 & 0 & 0 & 0 & 1  & 1  & 1\\
				0 & 0 & 0 & 0 & 1 & 0 & 0  & 1  & 1\\
			\end{array}\right)}$$whose rank is equal to $6$.
	\end{example}
	
	\section{Proofs}
	\subsection{Proof of $\mathcal{K}(P)\subseteq\mathcal{J}(P)$}
	Let $1\leq\nu\leq\mu\leq n$ and $j\leq k$ in $[1,m-3]$ 
	such that $i_{2\nu-1}+1 \leq j \leq i_{2\nu}$ and $i_{2\mu}+1 \leq k \leq i_{2\mu+1}$.
	We want to find a $(e_1(P)+1)\times (e_1(P)+1)$ submatrix of $N_1(P)$ whose determinant is equal to $\pm \widehat{D^P_{j,k}}$.
	
	\begin{definition}
		To get the desired submatrix,
		\begin{enumerate}
			\item remove every column of $N_1(P)$ labeled by 
			$$(\{\gamma : 1 \leq \beta \leq n \text{ and } i_{2\beta}+2 \leq \gamma \leq i_{2\beta+1}\}\cup\{m-2\})-\{k+1\};$$
			\item remove every row from $N_1(P)$ labeled by 
			$$\{\underline{\gamma} : 1 \leq \beta \leq n \text{ and } i_{2\beta-1}+1 \leq \gamma \leq i_{2\beta}-1\}
			-\{\underline{j}\}.$$
		\end{enumerate}
	\end{definition}

	This submatrix of $N_1(P)$ will be referred to as $A^P_{j,k}$.
	
	\begin{example}\label{impex}
		Let $m-3=16$ and consider $P=(0,2,3,3,4,6,7,7,8,10,12,14)\in\mathcal I$.
		Here, $e_2(P)=8$, $e_1(P)=8+6=14$, and the matrix obtained by removing all zero columns from $N_1(P)$ is
		$${\tiny
			\left(
			\begin{array}{ccccccccccccccccc}
				1 & 0 & x_{1,3} & x_{1,4} & 0       & 0 & x_{1,7} & x_{1,8} & 0       & 0       & x_{1,11}  & x_{1,12}  & 0         & 0        & x_{1,15}  & x_{1,16}  & 0\\
				0 & 1 & x_{2,3} & x_{2,4} & 0       & 0 & x_{2,7} & x_{2,8} & 0       & 0       & x_{2,11}  & x_{2,12}  & 0         & 0        & x_{2,15}  & x_{2,16}  & 0\\
				0 & 0 & 0       & 0       & 1       & 0 & x_{5,7} & x_{5,8} & 0       & 0       & x_{5,11}  & x_{5,12}  & 0         & 0        & x_{5,15}  & x_{5,16}  & 0\\
				0 & 0 & 0       & 0       & 0       & 1 & x_{6,7} & x_{6,8} & 0       & 0       & x_{6,11}  & x_{6,12}  & 0         & 0        & x_{6,15}  & x_{6,16}  & 0\\
				0 & 0 & 0       & 0       & 0       & 0 & 0       &       0 & 1       & 0       & x_{9,11}  & x_{9,12}  & 0         & 0        & x_{9,15}  & x_{9,16}  & 0\\
				0 & 0 & 0       & 0       & 0       & 0 & 0       &       0 & 0       & 1       & x_{10,11} & x_{10,12} & 0         & 0        & x_{10,15} & x_{10,16} & 0\\
				0 & 0 & 0       & 0       & 0       & 0 & 0       &       0 & 0       & 0       & 0         & 0         & 1         & 0        & x_{13,15} & x_{13,16} & 0\\
				0 & 0 & 0       & 0       & 0       & 0 & 0       &       0 & 0       & 0       & 0         & 0         & 0         & 1        & x_{14,15} & x_{14,16} & 0\\
				0 & 1 & 0       & x_{1,3} & x_{1,4} & 0 & 0       & x_{1,7} & x_{1,8} & 0       & 0         & x_{1,11}  & x_{1,12}  & 0        & 0         & x_{1,15}  & x_{1,16}\\
				0 & 0 & 1       & x_{2,3} & x_{2,4} & 0 & 0       & x_{2,7} & x_{2,8} & 0       & 0         & x_{2,11}  & x_{2,12}  & 0        & 0         & x_{2,15}  & x_{2,16}\\
				0 & 0 & 0       & 0       & 0       & 1 & 0       & x_{5,7} & x_{5,8} & 0       & 0         & x_{5,11}  & x_{5,12}  & 0        & 0         & x_{5,15}  & x_{5,16}\\
				0 & 0 & 0       & 0       & 0       & 0 & 1       & x_{6,7} & x_{6,8} & 0       & 0         & x_{6,11}  & x_{6,12}  & 0        & 0         & x_{6,15}  & x_{6,16}\\
				0 & 0 & 0       & 0       & 0       & 0 & 0       & 0       &       0 & 1       & 0         & x_{9,11}  & x_{9,12}  & 0        & 0         & x_{9,15}  & x_{9,16}\\
				0 & 0 & 0       & 0       & 0       & 0 & 0       & 0       &       0 & 0       & 1         & x_{10,11} & x_{10,12} & 0        & 0         & x_{10,15} & x_{10,16}\\
				0 & 0 & 0       & 0       & 0       & 0 & 0       & 0       &       0 & 0       & 0         & 0         & 0         & 1        & 0         & x_{13,15} & x_{13,16}\\
				0 & 0 & 0       & 0       & 0       & 0 & 0       & 0       &       0 & 0       & 0         & 0         & 0         & 0        & 1         & x_{14,15} & x_{14,16}\\
				0 & 0 & 0       & 1       & 0       & 0 & 0       & 0       &       0 & 0       & 0         & y_{4,12}  & 0         & 0        & 0         & y_{4,16}  & y_{4,17}\\
				0 & 0 & 0       & 0       & 0       & 0 & 0       & 1       &       0 & 0       & 0         & y_{8,12}  & 0         & 0        & 0         & y_{8,16}  & y_{8,17}\\
			\end{array}\right)}.$$
		Consider $\mathcal{L}^P(5,11)=\{x_{5,11},x_{6,12}\}$.
		Now we remove all the columns labeled by 
		$$\{\gamma : 1 \leq \beta \leq 6\text{ and } i_{2\beta}+2 \leq \gamma \leq i_{2\beta+1}\}-\{12\}=\{16,17\}$$ 
		to get
		$${\tiny
			\left(
			\begin{array}{ccccccccccccccccc}
				1 & 0 & x_{1,3} & x_{1,4} & 0       & 0 & x_{1,7} & x_{1,8} & 0       & 0       & x_{1,11}  & x_{1,12}  & 0         & 0        & x_{1,15}  \\
				0 & 1 & x_{2,3} & x_{2,4} & 0       & 0 & x_{2,7} & x_{2,8} & 0       & 0       & x_{2,11}  & x_{2,12}  & 0         & 0        & x_{2,15}  \\
				0 & 0 & 0       & 0       & 1       & 0 & x_{5,7} & x_{5,8} & 0       & 0       & x_{5,11}  & x_{5,12}  & 0         & 0        & x_{5,15}  \\
				0 & 0 & 0       & 0       & 0       & 1 & x_{6,7} & x_{6,8} & 0       & 0       & x_{6,11}  & x_{6,12}  & 0         & 0        & x_{6,15}  \\
				0 & 0 & 0       & 0       & 0       & 0 & 0       &       0 & 1       & 0       & x_{9,11}  & x_{9,12}  & 0         & 0        & x_{9,15}  \\
				0 & 0 & 0       & 0       & 0       & 0 & 0       &       0 & 0       & 1       & x_{10,11} & x_{10,12} & 0         & 0        & x_{10,15} \\
				0 & 0 & 0       & 0       & 0       & 0 & 0       &       0 & 0       & 0       & 0         & 0         & 1         & 0        & x_{13,15} \\
				0 & 0 & 0       & 0       & 0       & 0 & 0       &       0 & 0       & 0       & 0         & 0         & 0         & 1        & x_{14,15} \\
				0 & 1 & 0       & x_{1,3} & x_{1,4} & 0 & 0       & x_{1,7} & x_{1,8} & 0       & 0         & x_{1,11}  & x_{1,12}  & 0        & 0         \\
				0 & 0 & 1       & x_{2,3} & x_{2,4} & 0 & 0       & x_{2,7} & x_{2,8} & 0       & 0         & x_{2,11}  & x_{2,12}  & 0        & 0         \\
				0 & 0 & 0       & 0       & 0       & 1 & 0       & x_{5,7} & x_{5,8} & 0       & 0         & x_{5,11}  & x_{5,12}  & 0        & 0         \\
				0 & 0 & 0       & 0       & 0       & 0 & 1       & x_{6,7} & x_{6,8} & 0       & 0         & x_{6,11}  & x_{6,12}  & 0        & 0         \\
				0 & 0 & 0       & 0       & 0       & 0 & 0       & 0       &       0 & 1       & 0         & x_{9,11}  & x_{9,12}  & 0        & 0         \\
				0 & 0 & 0       & 0       & 0       & 0 & 0       & 0       &       0 & 0       & 1         & x_{10,11} & x_{10,12} & 0        & 0         \\
				0 & 0 & 0       & 0       & 0       & 0 & 0       & 0       &       0 & 0       & 0         & 0         & 0         & 1        & 0         \\
				0 & 0 & 0       & 0       & 0       & 0 & 0       & 0       &       0 & 0       & 0         & 0         & 0         & 0        & 1         \\
				0 & 0 & 0       & 1       & 0       & 0 & 0       & 0       &       0 & 0       & 0         & y_{4,12}  & 0         & 0        & 0         \\
				0 & 0 & 0       & 0       & 0       & 0 & 0       & 1       &       0 & 0       & 0         & y_{8,12}  & 0         & 0        & 0         \\
			\end{array}\right)}.$$
		
		Now remove all the rows labeled by
		$$\{\underline{\gamma} : 1 \leq \beta \leq 6 \text{ and } i_{2\beta-1}+1 \leq \gamma \leq i_{2\beta}-1\}-\{\underline{5}\}=
		\{\underline{1},\underline{9},\underline{13}\}$$
		to get
		$$A^P_{5,11}={\tiny
			\left(
			\begin{array}{ccccccccccccccccc}
				1 & 0 & x_{1,3} & x_{1,4} & 0       & 0 & x_{1,7} & x_{1,8} & 0       & 0       & x_{1,11}  & x_{1,12}  & 0         & 0        & x_{1,15}  \\
				0 & 1 & x_{2,3} & x_{2,4} & 0       & 0 & x_{2,7} & x_{2,8} & 0       & 0       & x_{2,11}  & x_{2,12}  & 0         & 0        & x_{2,15}  \\
				0 & 0 & 0       & 0       & 1       & 0 & x_{5,7} & x_{5,8} & 0       & 0       & x_{5,11}  & x_{5,12}  & 0         & 0        & x_{5,15}  \\
				0 & 0 & 0       & 0       & 0       & 1 & x_{6,7} & x_{6,8} & 0       & 0       & x_{6,11}  & x_{6,12}  & 0         & 0        & x_{6,15}  \\
				0 & 0 & 0       & 0       & 0       & 0 & 0       &       0 & 1       & 0       & x_{9,11}  & x_{9,12}  & 0         & 0        & x_{9,15}  \\
				0 & 0 & 0       & 0       & 0       & 0 & 0       &       0 & 0       & 1       & x_{10,11} & x_{10,12} & 0         & 0        & x_{10,15} \\
				0 & 0 & 0       & 0       & 0       & 0 & 0       &       0 & 0       & 0       & 0         & 0         & 1         & 0        & x_{13,15} \\
				0 & 0 & 0       & 0       & 0       & 0 & 0       &       0 & 0       & 0       & 0         & 0         & 0         & 1        & x_{14,15} \\
				0 & 0 & 1       & x_{2,3} & x_{2,4} & 0 & 0       & x_{2,7} & x_{2,8} & 0       & 0         & x_{2,11}  & x_{2,12}  & 0        & 0         \\
				0 & 0 & 0       & 0       & 0       & 1 & 0       & x_{5,7} & x_{5,8} & 0       & 0         & x_{5,11}  & x_{5,12}  & 0        & 0         \\
				0 & 0 & 0       & 0       & 0       & 0 & 1       & x_{6,7} & x_{6,8} & 0       & 0         & x_{6,11}  & x_{6,12}  & 0        & 0         \\
				0 & 0 & 0       & 0       & 0       & 0 & 0       & 0       &       0 & 0       & 1         & x_{10,11} & x_{10,12} & 0        & 0         \\
				0 & 0 & 0       & 0       & 0       & 0 & 0       & 0       &       0 & 0       & 0         & 0         & 0         & 0        & 1         \\
				0 & 0 & 0       & 1       & 0       & 0 & 0       & 0       &       0 & 0       & 0         & y_{4,12}  & 0         & 0        & 0         \\
				0 & 0 & 0       & 0       & 0       & 0 & 0       & 1       &       0 & 0       & 0         & y_{8,12}  & 0         & 0        & 0         \\
			\end{array}\right)},$$
		a $15 \times 15$ matrix.
	\end{example}
	
	\begin{remark}
		Notes about $A^P_{j,k}$:
		\begin{enumerate}
			\item it is a $(e_1(P)+1) \times (e_1(P)+1)$ matrix,
			\item when we want to remember the label of a column, a row or a position of $A^P_{j,k}$, 
			we will keep the labeling of $N_1(P)$,
			\begin{comment}

				\item for $\beta \in [1,n]$, 
				\begin{enumerate}
					
					\item[a)] the column $(i_{2\beta}+1)$ looks like $$
					\begin{pmatrix}
						x_{i_1+1,i_{2\beta}+1}&
						\dots&
						x_{i_2,i_{2\beta}+1}&
						\cdots&
						x_{i_{2\beta-1}+1,i_{2\beta}+1}&
						\cdots&
						x_{i_{2\beta},i_{2\beta}+1}&
						0&
						\cdots\ 0\ 0\ 0\ \cdots\ 
						1&0\ \cdots\ 0
					\end{pmatrix}^T$$
					\item[b)] \text{and the column $(i_{2\beta+1}+1)$ looks like }
				\end{enumerate}
				$$
				\begin{pmatrix}
					0\ 
					\cdots\   0 \ 
					\cdots \ 
					1\ 
					\cdots\ 
					0\ 
					0\ 
					\cdots&
					0&
					x_{i_2,i_{2\beta+1}}&
					x_{i_4,i_{2\beta+1}}&
					\cdots&
					x_{i_{2\beta},i_{2\beta+1}}&
					0&
					\cdots&
					0
				\end{pmatrix}^T$$
				\item Notice that the 1-entry of column labeled by $(i_{2\beta}+1)$ lies in 
				the row labeled by $\underline{i_{2\beta}}$ and
				the 1-entry of the column labeled by $(i_{2\beta+1}+1)$ lies in 
				the row labeled by $\overline{i_{2\beta-1}+1}$.
				Every other column has only one 1-entry and the rest are 0 entries, 
				with the exception of the column labeled by $i_{2\nu}$ which has two 1-entries and the rest are 0 entries.
			\end{comment}
		\end{enumerate}
	\end{remark}

	\begin{case}
		Assume that $\mathcal{A}(P)=\emptyset.$ Then
		\label{det1}
		$\det(A^P_{j,k})=\pm D^P_{j,k}$.
	\end{case}
	%The definitions of $\underline{A_{j,k}}(S_1;S_2)$ and $\overline{A_{j,k}}(S_1;S_2)$ are clear.
	\begin{proof}%[Proof of Lemma \ref{det1}]
		Let $\nu \leq m_1 \leq \mu$.
		By induction, it suffices to show that $\det(A_{j,k})$ contains
		$$\pm \left(\varphi_{1}^{(j,k)}(\mu)\varphi_{1}^{(j,k)}(\widetilde{\mu})+\varphi_{1}^{(j,k)}(\mu,m_1)\varphi_{1}^{(j,k)}(\widetilde{\mu,m_1})-
		\varphi_{2}^{(j,k)}(\mu)\varphi_{2}^{(j,k)}(\widetilde{\mu})-\varphi_{2}^{(j,k)}(\mu,m_1)\varphi_{2}^{(j,k)}(\widetilde{\mu,m_1})\right).$$
		
		To see that this is indeed the case, we will use cofactor expansion along 
		the column labeled by $(k+1)$.
		
		Remember that $\mathcal{A}(P)=\emptyset.$
		In the case where $k<m-3$, the cofactor of $x(\alpha,k+1)$ for $1 \leq \alpha \leq i_{2\nu}$ 
		such that 
		$$\alpha \notin \{i_{2\beta-1}+1: 2 \leq \beta \leq \nu\} \cup \{j+1\}$$ is equal to 0 
		since the resulting minor has a  zero column.
		The cofactor of $x(\alpha,k+1)$ for $$\alpha \in \{i_{2\beta-1}+1 : 2 \leq \beta \leq \nu\}$$
		is equal to 0 since the resulting minor contains a set of linearly dependent columns labeled by
		$$\left\{\overline{j+1}\right\} \cup \left\{\overline{i_{2\beta}+1}: \nu \leq \beta \leq n-1\right\}
		\cup \left\{\underline{j}\right\} \cup \left\{\underline{i_{2\beta}}: \nu \leq \beta \leq n\right\}:$$ 
		notice that the row labeled by $\overline{j+1}$ has a 1-entry at the column labeled by $(j+1)$,
		variables in the column labeled by $(i_{2\beta}+1)$ for every $\nu < \beta \leq n$, and 0 everywhere else.
		Furthermore the row labeled by $\underline{j}$ has a 1-entry at the column labeled by $(j+1)$, variables in the
		column labeled by $i_{2\beta}$ for every $\nu < \beta \leq n$, and 0 everywhere else. 
		We can perform row operations using the rows labeled by 
		$$\left\{\overline{j+1}\right\} \cup \left\{\overline{i_{2\beta}+1}: \nu \leq \beta \leq n\right\} \cup
		\left\{\underline{j}\right\} \cup \left\{\underline{i_{2\beta}}: \nu \leq \beta \leq n\right\}$$ so that the row labeled by $\overline{j+1}$
		and the row labeled by $\underline{j}$ have a 1-entry in the column labeled by $(j+1)$ and 0 everywhere else.
		This argument also works for justifying that the cofactor of $x(i_{2\beta},k)$ where $1 \leq \beta \leq \nu-1$ is 0, 
		regardless of whether or not $k=m-3$.

		\begin{quote}
			To illustrate, consider the case where $P = (0,2,4,6,8,10,12,14) \in \mathcal{I}$
			$$A^P_{5,11}={\tiny
				\left(
				\begin{array}{ccccccccccccccccc}
					1 & 0 & x_{1,3} & 0       & 0 & x_{1,7} & 0       & 0       & x_{1,11}  & x_{1,12}  & 0         & 0        & x_{1,15}  \\
					0 & 1 & x_{2,3} & 0       & 0 & x_{2,7} & 0       & 0       & x_{2,11}  & x_{2,12}  & 0         & 0        & x_{2,15}  \\
					0 & 0 & 0       & 1       & 0 & x_{5,7} & 0       & 0       & x_{5,11}  & x_{5,12}  & 0         & 0        & x_{5,15}  \\
					0 & 0 & 0       & 0       & 1 & x_{6,7} & 0       & 0       & x_{6,11}  & x_{6,12}  & 0         & 0        & x_{6,15}  \\
					0 & 0 & 0       & 0       & 0 & 0       & 1       & 0       & x_{9,11}  & x_{9,12}  & 0         & 0        & x_{9,15}  \\
					0 & 0 & 0       & 0       & 0 & 0       & 0       & 1       & x_{10,11} & x_{10,12} & 0         & 0        & x_{10,15} \\
					0 & 0 & 0       & 0       & 0 & 0       & 0       & 0       & 0         & 0         & 1         & 0        & x_{13,15} \\
					0 & 0 & 0       & 0       & 0 & 0       & 0       & 0       & 0         & 0         & 0         & 1        & x_{14,15} \\
					0 & 0 & 1       & x_{2,4} & 0 & 0       & x_{2,8} & 0       & 0         & x_{2,11}  & x_{2,12}  & 0        & 0         \\
					0 & 0 & 0       & 0       & 1 & 0       & x_{5,8} & 0       & 0         & x_{5,11}  & x_{5,12}  & 0        & 0         \\
					0 & 0 & 0       & 0       & 0 & 1       & x_{6,8} & 0       & 0         & x_{6,11}  & x_{6,12}  & 0        & 0         \\
					0 & 0 & 0       & 0       & 0 & 0       &       0 & 0       & 1         & x_{10,11} & x_{10,12} & 0        & 0         \\
					0 & 0 & 0       & 0       & 0 & 0       &       0 & 0       & 0         & 0         & 0         & 0        & 1         \\
				\end{array}\right)}.$$ 
			Here, $(\nu,\mu)=(2,3)$, $n=4$, $(j,k)=(5,11)$, $m-3=16$, and the column labeled by 12 in $A^P_{5,11}$ is 
			$$\left(
			\begin{array}{ccccc}
				x_{1,12}&
				x_{2,12}&
				\cdots&
				x_{10,11}&
				0
			\end{array}\right)^T.$$
			By removing the row and column containing $x_{\alpha,12}$ where $\alpha \in \{1,2,10\}$,
			we get a matrix with where column $\alpha$ only has 0 entries. 
			(Remember that we are using the same labels used for the columns of $N_2(P)$.)
			By removing the row and column containing $x_{5,12}$ or $x_{2,11}$,
			we see that the rows labeled by $\overline{6},\ \overline{9},\ \underline{5},\ \underline{6},$ 
			and $\underline{10}$ are linearly dependent.
			Thus the cofactors of $x_{1,12}$, $x_{2,12}$, $x_{10,12}$, $x_{5,12}$, and $x_{2,11}$ are 0.
		\end{quote}
		
		The submatrix of $A^P_{j,k}$ obtained by removing the rows labeled by $S_1$ and the columns labeled by $S_2$
		will be written as $A^P_{j,k}(S_1;S_2)$. 
		
		To compute the cofactor of $x(j+1,k+1)$, observe that it lies 
		\begin{enumerate}
			\item in row $\left(\sum\limits_{\beta=1}^{\nu-1}i_{2\beta}-i_{2\beta-1}\right)+(j-i_{2\nu-1})$ of $A^P_{j,k}$
			since the height of $N^{(2)}_{\beta,\mu}$ is $i_{2\beta}-i_{2\beta-1}$ and $x(j+1,k+1)$ is in row $(j-i_{2\nu-1})$ of $N^{(2)}_{\nu,\mu}$
			and
			\item in column $\left[\left(\sum\limits_{\beta=1}^\mu i_{2\beta}-i_{2\beta-1}\right)+\mu+1\right]$ of $A^P_{j,k}$
			since the width of $N^{(2)}_{\nu,\beta}$ has been reduced to $i_{2\beta}-i_{2\beta-1}+1$ in $A^P_{j,k}$.
		\end{enumerate}
		
		Thus the cofactor of $x(j+1,k+1)$ is the product of $M^0=\det(A_{j,k}(\overline{j+1};k+1))$ and   $$r^0=\text{sgn}\left(\left(\sum\limits_{\beta=1}^{\nu-1}i_{2\beta}-i_{2\beta-1}\right)+(j-i_{2\nu-1})+
		\left(\sum\limits_{\beta=1}^\mu i_{2\beta}-i_{2\beta-1}\right)+\mu+1\right),$$
		where $\text{sgn}\colon\mathbb{Z}\longrightarrow \{1,-1\}$ is the standard group homomorphism. 
		To compute $M^0$, first take the cofactor expansion down the column labeled by $i_2+1$.
		It is not difficult to see that the only nonzero cofactor is the cofactor of 1 in the 
		$\sum\limits_{\beta=1}^{n}i_{2\beta}-i_{2\beta-1}$ row of $A^P_{j,k}(\overline{j+1};k+1)$.
		Thus $$M^0=\text{sgn}\left(\left[\sum\limits_{\beta=1}^{n}i_{2\beta}-i_{2\beta-1}\right]+\left[(i_2-i_1)+1\right]\right)
		\det(A^P_{j,k}(\underline{i_2},\overline{j+1};i_2+1,k+1)).$$
		To compute $\det(A^P_{j,k}(\underline{i_2},\overline{j+1};i_2+1,k+1))$, 
		take the cofactor expansion of down the column labeled by $i_4+1$.
		Again, the only nonzero cofactor is the cofactor of 1 in 
		the $\sum\limits_{\beta=1}^{n}i_{2\beta}-i_{2\beta-1}$ row and
		the $\sum\limits_{\beta=1}^2 i_{2\beta}-i_{2\beta-1}$ column of $A^P_{j,k}(\underline{i_2},\overline{j+1};i_2+1,k+1)$.
		Thus, 
		$$M^0=\text{sgn}\left(\sum\limits_{\ell=1}^2\left(\left[\sum\limits_{\beta=1}^{n}i_{2\beta}-i_{2\beta-1}\right]+
		\left[\left(\sum\limits_{\beta=1}^\ell i_{2\beta}-i_{2\beta-1}\right)+1\right]\right)\right)
		\det(A^P_{j,k}(\overline{j+1},\underline{i_2},\underline{i_4};k+1,i_2+1,i_4+1)).$$
		By continuing in this manner, we see that 
		$$M^0=\text{sgn}\left(\sum\limits_{\ell=1}^{\nu-1}\left(\left[\sum\limits_{\beta=1}^{n}i_{2\beta}-i_{2\beta-1}\right]+
		\left[\left(\sum\limits_{\beta=1}^\ell i_{2\beta}-i_{2\beta-1}\right)+1\right]\right)\right)\det(A^P_{j,k}(S_1;S_2)).$$
		where $S_1=\{\underline{i_2},\underline{i_4},\dots,\underline{i_{2(\nu-1)}},\overline{j+1}\}$ and
		$S_2=\{i_2+1,i_4+1,\dots,i_{2(\nu-1)}+1,k+1\}$.
		
		To compute $\det(A^P_{j,k}(S_1;S_2))$, take the cofactor expansion down the column labeled by $i_{2\nu}$.
		Since the only nonzero entry of the column labeled by $i_{2\nu}$ is its row labeled by $\underline{j}$. 
		Thus $$\det(A^P_{j,k}(S_1;S_2))=\text{sgn}\left(\sum\limits_{\beta=1}^{n}i_{2\beta}-i_{2\beta-1}+
		\sum\limits_{\beta=1}^{\nu-1}i_{2\beta}-i_{2\beta-1}+(j-i_{2\nu-1}+1)\right)
		\det(A^P_{j,k}(S_1\cup\{\underline{j}\};S_2\cup\{i_{2\nu}\})).$$
		To compute $\det(A^P_{j,k}(S_1\cup\{\underline{j}\};S_2\cup\{i_{2\nu}\})),$
		take the cofactor expansion down the column labeled by $i_{2\nu}+1$, whose only nonzero entry is in the row labeled by $\underline{i_{2\nu}}$.
		By continuing this process, we get that $M_0$ is the product of
		\begin{enumerate}
			\item $\text{sgn}\left(\sum\limits_{\ell=1}^{\nu-1}\left(
			\left[\sum\limits_{\beta=1}^{n}i_{2\beta}-i_{2\beta-1}\right]+
			\left[\sum\limits_{\beta=1}^\ell(i_{2\beta}-i_{2\beta-1})+1\right]\right)\right)$,
			\item $\text{sgn}\left(\left[\sum\limits_{\beta=1}^{n}i_{2\beta}-i_{2\beta-1}\right]+
			\left[\sum\limits_{\beta=1}^{\nu-1}(i_{2\beta}-i_{2\beta-1})+(j-i_{2\nu-1}+1)\right]\right)$,
			\item $\text{sgn}\left(\sum\limits_{\ell=1}^{\nu-1}\left(\left[\sum\limits_{\beta=1}^{n}i_{2\beta}-i_{2\beta-1}\right]+
			\left[\sum\limits_{\beta=1}^\ell i_{2\beta}-i_{2\beta-1}\right]\right)\right)\det(A^P_{j,k}(S_1';S_2'))$ 
			where $S_1'=S_1 \cup\{\underline{j},\underline{i_{2\nu}},\underline{i_{2(\nu+1)}},\dots,\underline{i_{2n}}\}$ and
			$S_2'=S_2 \cup \{i_{2\nu},i_{2\nu}+1,i_{2(\nu+1)}+1,\dots,i_{2n}+1\}$.
		\end{enumerate}   
		
		Since $A^P_{j,k}(S_1';S_2')$ is an upper triangular matrix with only 1 entries in the diagonal,
		$\det(A^P_{j,k}(S_1';S_2'))=1$.
		Thus $r^0M^0x(j+1,k+1)=r^0M^0\varphi_2^{(j,k)}(\mu)\varphi_2^{(j,k)}(\widetilde{\mu})$ is a term of $\det(A_{j,k})$.
		\begin{quote}
			To illustrate, let us return to our running example.
			
			Recall that $P = (0,2,4,6,8,10,12,14) \in \mathcal{I}$.
			The cofactor of $x_{6,12}$ is
			$$\text{sgn}\left(\left[\left(\sum\limits_{\beta=1}^1 i_{2\beta}-i_{2\beta-1}\right)+2\right]+
			\left[\left(\sum\limits_{\beta=1}^3 i_{2\beta}-i_{2\beta-1}\right)+4\right]\right)\cdot
			\det(A^P_{5,11}(\overline{6};12)).$$
			Note that 
			$$A^P_{5,11}(\overline{6};12)={\tiny
				\left(
				\begin{array}{ccccccccccccccccc}
					1 & 0 & x_{1,3} & 0       & 0 & x_{1,7} & 0       & 0       & x_{1,11}  & 0         & 0        & x_{1,15}  \\
					0 & 1 & x_{2,3} & 0       & 0 & x_{2,7} & 0       & 0       & x_{2,11}  & 0         & 0        & x_{2,15}  \\
					0 & 0 & 0       & 1       & 0 & x_{5,7} & 0       & 0       & x_{5,11}  & 0         & 0        & x_{5,15}  \\
					0 & 0 & 0       & 0       & 0 & 0       & 1       & 0       & x_{9,11}  & 0         & 0        & x_{9,15}  \\
					0 & 0 & 0       & 0       & 0 & 0       & 0       & 1       & x_{10,11} & 0         & 0        & x_{10,15} \\
					0 & 0 & 0       & 0       & 0 & 0       & 0       & 0       & 0         & 1         & 0        & x_{13,15} \\
					0 & 0 & 0       & 0       & 0 & 0       & 0       & 0       & 0         & 0         & 1        & x_{14,15} \\
					0 & 0 & 1       & x_{2,4} & 0 & 0       & x_{2,8} & 0       & 0         & x_{2,12}  & 0        & 0         \\
					0 & 0 & 0       & 0       & 1 & 0       & x_{5,8} & 0       & 0         & x_{5,12}  & 0        & 0         \\
					0 & 0 & 0       & 0       & 0 & 1       & x_{6,8} & 0       & 0         & x_{6,12}  & 0        & 0         \\
					0 & 0 & 0       & 0       & 0 & 0       &       0 & 0       & 1         & x_{10,12} & 0        & 0         \\
					0 & 0 & 0       & 0       & 0 & 0       &       0 & 0       & 0         & 0         & 0        & 1         \\
				\end{array}\right)}.
			$$
			To compute $\det(A^P_{5,11}(\overline{6};12))$, take the cofactor expansion down the column labeled by $i_2+1=3$.
			The only nonzero cofactor is of 1 in the $\sum\limits_{\beta=1}^4 i_{2\beta}-i_{2\beta-1}=8$ row.
			Thus
			$$M^0=\text{sgn}\left(\left[\sum\limits_{\beta=1}^4 i_{2\beta}-i_{2\beta-1}\right]+
			\left[\left(\sum\limits_{\beta=1}^1 i_{2\beta}-i_{2\beta-1}\right)+1\right]\right)\cdot
			\det(A^P_{5,11}(\underline{2},\overline{6};3,12)).$$
			To compute $\det(A^P_{5,11}(\underline{2},\overline{6};3,12))$, note that 
			$$A^P_{5,11}(\underline{2},\overline{6};3,12)={\tiny
				\left(
				\begin{array}{ccccccccccccccccc}
					1 & 0 & 0       & 0 & x_{1,7} & 0       & 0       & x_{1,11}  & 0         & 0        & x_{1,15}  \\
					0 & 1 & 0       & 0 & x_{2,7} & 0       & 0       & x_{2,11}  & 0         & 0        & x_{2,15}  \\
					0 & 0 & 1       & 0 & x_{5,7} & 0       & 0       & x_{5,11}  & 0         & 0        & x_{5,15}  \\
					0 & 0 & 0       & 0 & 0       & 1       & 0       & x_{9,11}  & 0         & 0        & x_{9,15}  \\
					0 & 0 & 0       & 0 & 0       & 0       & 1       & x_{10,11} & 0         & 0        & x_{10,15} \\
					0 & 0 & 0       & 0 & 0       & 0       & 0       & 0         & 1         & 0        & x_{13,15} \\
					0 & 0 & 0       & 0 & 0       & 0       & 0       & 0         & 0         & 1        & x_{14,15} \\
					0 & 0 & 0       & 1 & 0       & x_{5,8} & 0       & 0         & x_{5,12}  & 0        & 0         \\
					0 & 0 & 0       & 0 & 1       & x_{6,8} & 0       & 0         & x_{6,12}  & 0        & 0         \\
					0 & 0 & 0       & 0 & 0       &       0 & 0       & 1         & x_{10,12} & 0        & 0         \\
					0 & 0  & 0       & 0 & 0       &       0 & 0       & 0         & 0         & 0        & 1         \\
				\end{array}\right)}.
			$$
			Take the cofactor expansion down the column labeled by $i_4=6$. 
			(This is column 4 of $A^P_{5,11}(\underline{2},\overline{6};3,12)$.)
			The only nonzero cofactor is clearly the cofactor of 1 
			which is on the row labeled by $\underline{j}=\underline{5}$.
			(This is row 8 of $A^P_{5,11}(\underline{2},\overline{6};3,12)$.)
			Thus $$\det(A^P_{5,11}(\underline{2},\overline{6};3,12))=
			\text{sgn}\left(\sum\limits_{\beta=1}^3 i_{2\beta}-i_{2\beta-1}+
			\left[\left(\sum\limits_{\beta=1}^1 i_{2\beta}-i_{2\beta-1}\right)+2\right]\right)\cdot
			\det(A^P_{5,11}(\underline{2},\underline{5},\overline{6};3,6,12)).$$
			
			To compute $\det(A^P_{5,11}(\underline{2},\underline{5},\overline{6};3,6,12))$, note that
			$$A^P_{5,11}(\underline{2},\underline{5},\overline{6};3,6,12)={\tiny
				\left(
				\begin{array}{ccccccccccccccccc}
					1 & 0 & 0       & x_{1,7} & 0       & 0       & x_{1,11}  & 0         & 0        & x_{1,15}  \\
					0 & 1 & 0       & x_{2,7} & 0       & 0       & x_{2,11}  & 0         & 0        & x_{2,15}  \\
					0 & 0 & 1       & x_{5,7} & 0       & 0       & x_{5,11}  & 0         & 0        & x_{5,15}  \\
					0 & 0 & 0       & 0       & 1       & 0       & x_{9,11}  & 0         & 0        & x_{9,15}  \\
					0 & 0 & 0       & 0       & 0       & 1       & x_{10,11} & 0         & 0        & x_{10,15} \\
					0 & 0 & 0       & 0       & 0       & 0       & 0         & 1         & 0        & x_{13,15} \\
					0 & 0 & 0       & 0       & 0       & 0       & 0         & 0         & 1        & x_{14,15} \\
					0 & 0 & 0       & 1       & x_{6,8} & 0       & 0         & x_{6,12}  & 0        & 0         \\
					0 & 0 & 0       & 0       &       0 & 0       & 1         & x_{10,12} & 0        & 0         \\
					0 & 0  & 0      & 0       &       0 & 0       & 0         & 0         & 0        & 1         \\
				\end{array}\right)}.
			$$
			Take the cofactor expansion down the column labeled by $i_4+1=4$.
			(This is column 4 of $A^P_{5,11}(\underline{2},\overline{6};3,12)$.)
			We see that the only nonzero cofactor is the cofactor of 1 which lies on the row labeled by $\underline{i_4}$.
			(The row containing $x(2,4)$.)
			Thus $$\det(A^P_{5,11}(\underline{2},\underline{5},\overline{6};3,6,12))=
			\text{sgn}\left(\sum\limits_{\beta=1}^4 i_{2\beta}-i_{2\beta-1}+
			\sum\limits_{\beta=1}^2 i_{2\beta}-i_{2\beta-1}\right)\cdot
			\det(A^P_{5,11}(\underline{2},\underline{5},\overline{6},\underline{6};3,6,7,12)).$$
			By repeating the same approach, we see that 
			$$M^0=
			\text{sgn}\left(\sum\limits_{\ell=2}^{4}\left(\left[\sum\limits_{\beta=1}^4 i_{2\beta}-i_{2\beta-1}\right]+
			\left[\sum\limits_{\beta=1}^\ell i_{2\beta}-i_{2\beta-1}\right]\right)\right)\cdot
			\det(A^P_{5,11}(S_1';S_2')).$$
			where $S_1'=\{\underline{2},\underline{5},\overline{6},\underline{8},\underline{10},\underline{14}\}$
			and $S_2'=\{3,6,7,11,12,15\}.$
			Since $A^P_{5,11}(S_1';S_2')$ is the identity matrix, $\det(A^P_{5,11}(S_1';S_2'))=1$.
			Thus
			{\tiny$$\text{sgn}\left(
				\left[\sum\limits_{\beta=1}^1 (i_{2\beta}-i_{2\beta-1})+2\right]+
				\left[\sum\limits_{\beta=1}^3 (i_{2\beta}-i_{2\beta-1})+4\right]+
				\sum\limits_{\ell=2}^{4}\left(
				\left[\sum\limits_{\beta=1}^4 i_{2\beta}-i_{2\beta-1}\right]+
				\left[\sum\limits_{\beta=1}^\ell i_{2\beta}-i_{2\beta-1}\right]\right)\right)x_{6,12}$$}
			is a term of $\det(A^P_{5,11}).$
		\end{quote}    
		Returning to the proof, a similar approach shows that the cofactor of $x(j,k)$ is 
		$$r_0=\text{sgn}\left(\left[\left(\sum\limits_{\beta=1}^n i_{2\beta}-i_{2\beta-1}\right)+\nu\right]+
		\left[\left(\sum\limits_{\beta=1}^\mu i_{2\beta}-i_{2\beta-1}\right)+\mu+1\right]\right)$$ 
		times $M_0=\det(A^P_{j,k}(\underline{j},k+1))$, which is equal to 
		$$\text{sgn}\left(
		\left[\sum\limits_{\ell=1}^n\left(\left(\sum\limits_{\beta=1}^n i_{2\beta}-i_{2\beta-1}\right)+1\right]+
		\left[\left(\sum\limits_{\beta=1}^\ell i_{2\beta}-i_{2\beta-1}\right)+1\right]\right)\right).$$
		Thus $$r_0M_0x(j,k)=r_0M_0\varphi_1^{(j,k)}(\widetilde{\mu})\varphi_1^{(j,k)}(\mu)$$ 
		is a term in $\det(A^P_{j,k})$.
		
		Note that $r_0M_0x(j,k)$ and $r^0M^0x(j+1,k+1)$ have opposite signs:
		\begin{enumerate}
			\item 
			\begin{align*}
				\dfrac{r^0}{r_0}&=\dfrac{\text{sgn}\left(\left[\left(\sum\limits_{\beta=1}^{\nu-1}i_{2\beta}-i_{2\beta-1}\right)+(j-i_{2\nu-1})\right]+
					\left[\left(\sum\limits_{\beta=1}^\mu i_{2\beta}-i_{2\beta-1}\right)+\mu+1\right]\right)}{(-1)^{\left[\left(\sum\limits_{\beta=1}^n i_{2\beta}-i_{2\beta-1}\right)+\nu\right]+
						\left[\left(\sum\limits_{\beta=1}^\mu i_{2\beta}-i_{2\beta-1}\right)+\mu+1\right]}}\\
				&=\dfrac{\text{sgn}\left(\left(\sum\limits_{\beta=1}^{\nu-1}i_{2\beta}-i_{2\beta-1}\right)+(j-i_{2\nu-1})\right)}{\text{sgn}\left(\left(\sum\limits_{\beta=1}^n i_{2\beta}-i_{2\beta-1}\right)+\nu\right)}
			\end{align*}
			\item 
			\begin{align*}
				\dfrac{M^0}{M_0}&=(-1)^{\left(\sum\limits_{\beta=1}^{\nu-1} -1 \right)+
					\left(\left[\sum\limits_{\beta=1}^{n} i_{2\beta}-i_{2\beta-1} \right]+
					\left[\sum\limits_{\beta=1}^{\nu-1}i_{2\beta}-i_{2\beta-1}+(j-i_{2\nu-1})\right]\right)+\left(\sum\limits_{\beta=\nu}^{n} -2 \right)}\\
				&=(-1)^{-(\nu-1)+\left(\left[\sum\limits_{\beta=1}^{n} i_{2\beta}-i_{2\beta-1} \right]+
					\left[\sum\limits_{\beta=1}^{\nu-1}i_{2\beta}-i_{2\beta-1}+(j-i_{2\nu-1})\right]\right)}
			\end{align*}
		\end{enumerate}
		Thus $\dfrac{r^0}{r_0} \cdot \dfrac{M^0}{M_0}=-1$.

		\begin{quote}
			Returning to our running example, it can be easily computed that the cofactor of $x_{5,11}$ is the product of 
			\begin{enumerate}
				\item $r_0=\text{sgn}\left(\left[\left(\sum\limits_{\beta=1}^4 i_{2\beta}-i_{2\beta-1}\right)+2\right]+
				\left[\left(\sum\limits_{\beta=1}^3 i_{2\beta}-i_{2\beta-1}\right)+4\right]\right)$
				\item $M_0=\det(A^P_{5,11}(\underline{5};12))$.
			\end{enumerate}
			Notice 
			$$A^P_{5,11}(\underline{5};12)={\tiny
				\left(
				\begin{array}{cccccccccccccccc}
					1 & 0 & x_{1,3} & 0       & 0 & x_{1,7} & 0       & 0       & x_{1,11}  & 0         & 0        & x_{1,15}  \\
					0 & 1 & x_{2,3} & 0       & 0 & x_{2,7} & 0       & 0       & x_{2,11}  & 0         & 0        & x_{2,15}  \\
					0 & 0 & 0       & 1       & 0 & x_{5,7} & 0       & 0       & x_{5,11}  & 0         & 0        & x_{5,15}  \\
					0 & 0 & 0       & 0       & 1 & x_{6,7} & 0       & 0       & x_{6,11}  & 0         & 0        & x_{6,15}  \\
					0 & 0 & 0       & 0       & 0 & 0       & 1       & 0       & x_{9,11}  & 0         & 0        & x_{9,15}  \\
					0 & 0 & 0       & 0       & 0 & 0       & 0       & 1       & x_{10,11} & 0         & 0        & x_{10,15} \\
					0 & 0 & 0       & 0       & 0 & 0       & 0       & 0       & 0         & 1         & 0        & x_{13,15} \\
					0 & 0 & 0       & 0       & 0 & 0       & 0       & 0       & 0         & 0         & 1        & x_{14,15} \\
					0 & 0 & 1       & x_{2,4} & 0 & 0       & x_{2,8} & 0       & 0         & x_{2,12}  & 0        & 0         \\
					0 & 0 & 0       & 0       & 0 & 1       & x_{6,8} & 0       & 0         & x_{6,12}  & 0        & 0         \\
					0 & 0 & 0       & 0       & 0 & 0       &       0 & 0       & 1         & x_{10,12} & 0        & 0         \\
					0 & 0 & 0       & 0       & 0 & 0       &       0 & 0       & 0         & 0         & 0        & 1         \\
				\end{array}\right)}$$
			whose determinant is
			$$\text{sgn}\left(\sum\limits_{\ell=1}^4\left(\left[\left(\sum\limits_{\beta=1}^4 i_{2\beta}-i_{2\beta-1}\right)+1\right]+
			\left[\left(\sum\limits_{\beta=1}^\ell i_{2\beta}-i_{2\beta-1}\right)+1\right]\right)\right).$$
			Thus $$r_0M_0\varphi_1^{(2,3)}(\widetilde{(3)})\varphi_1^{(2,3)}((3))$$ 
			is a term in $\det(A^P_{5,11})$.
			Furthermore, its sign is opposite of $r^0M^0\varphi_2^{(2,3)}(\widetilde{(3)})\varphi_1^{(2,3)}(3)$.
		\end{quote}    
		Returning to the proof, now let $\nu \leq m_1 < \mu=m_0$.
		The cofactor of $\varphi_2^{(j,k)}(m_0,m_1)=x(i_{2m_1+1}+1,k+1)$ is the product of 
		\begin{enumerate}
			\item $s^1=\text{sgn}\left(\left[\left(\sum\limits_{\beta=1}^{m_1} i_{2\beta}-i_{2\beta-1}\right)+1\right]+
			\left[\left(\sum\limits_{\beta=1}^\mu i_{2\beta}-i_{2\beta-1}\right)+\mu+1\right]\right)$,
			\item $\det(A^P_{j,k}(\overline{i_{2m_1}+1};k+1)$
		\end{enumerate}
		To compute $\det(A^P_{j,k}(\overline{i_{2m_1+1}+1};k+1)$, 
		take the cofactor expansion down the column labeled by $i_{2m_1}+1$.
		Notice that the cofactor of $x(i_{2\alpha},i_{2m_1})$ for $1 \leq \alpha \leq \nu-1$ is 0 for the same reasoning
		that the cofactor of $x(i_{2\alpha},k)$ is 0.
		
		However, the cofactor of $\varphi_2^{(j,k)}(\widetilde{(m_0,m_1)})=x(j,i_{2m_1+1})$ is the product of
		\begin{enumerate}
			\item $r^1=\text{sgn}\left(\left[\left(\sum\limits_{\beta=1}^{n} i_{2\beta}-i_{2\beta-1}\right)+\nu-1\right]+
			\left[\left(\sum\limits_{\beta=1}^{m_1} i_{2\beta}-i_{2\beta-1}\right)+m_1+1\right]\right)$ and
			\item $M^1=\det(A^P_{j,k}(\overline{i_{2m_1+1}+1},\underline{j};i_{2m_1+1}+1,k+1)).$
		\end{enumerate}
		To compute $M^1$, we will take the cofactor expansion of $A^P_{j,k}(\overline{i_{2m_1+1}+1},\underline{j};i_{2m_1+1}+1,k+1)$
		down the column labeled by $i_{2\alpha}+1$ for $1 \leq \alpha \leq n$ 
		in the order as $\alpha$ increases to get that $M^1$
		is the product of 
		\begin{enumerate}
			\item $\text{sgn}\left(\sum\limits_{\ell=1}^{m_1}
			\left(\left[\sum\limits_{\beta=1}^n i_{2\beta}-i_{2\beta-1}\right]+
			\left[\left(\sum\limits_{\beta=1}^\ell i_{2\beta}-i_{2\beta-1}\right)+1\right]\right)\right)$ and
			\item $\text{sgn}\left(\sum\limits_{\ell=m_1+1}^{n}
			\left(\left[\sum\limits_{\beta=1}^n i_{2\beta}-i_{2\beta-1}\right]+
			\left[\sum\limits_{\beta=1}^\ell i_{2\beta}-i_{2\beta-1}\right]\right)\right).$
		\end{enumerate}
		Thus $$s^1r^1M^1x(i_{2m_1+1}+1,k+1)x(j,i_{2m_1+1})=s^1r^1M^1\varphi_2^{(j,k)}(m_0,m_1)\varphi_2^{(j,k)}(\widetilde{(m_0,m_1)})$$
		is a term in $\det(A_{j,k})$.
		
		Similar reasoning shows that the cofactor of $x(i_{2m_1},k) \in \underline{A^P_{j,k}}$ is the product of
		\begin{enumerate}
			\item $s_1=\text{sgn}\left(\left[\left(\sum\limits_{\beta=1}^n i_{2\beta}-i_{2\beta-1}\right)+m_1+1\right]+
			\left[\left(\sum\limits_{\beta=1}^\mu i_{2\beta}-i_{2\beta-1}\right)+\mu+1\right]\right)$ and
			\item $\det(A^P_{j,k}(\underline{i_{2m_1}};k+1))$. 
		\end{enumerate}
		By taking the cofactor expansion along the column labeled by $i_{2m_1+1}$, 
		we see that the cofactor of $x(j+1,i_{2m_1}+1)$ and $x(i_{2\alpha+1}+1,i_{2m_1}+1)$ 
		for $\nu \leq \alpha < m_1$ are the only nonzero cofactors,
		using a similar argument used for computing the cofactor of $x(j,k)$ along the column labeled by $(k+1)$.
		The cofactor of $x(j+1,i_{2m_1}+1)$ is equal to
		$$r_1=\text{sgn}\left(\left[\left(\sum\limits_{\beta=1}^{\nu-1} i_{2\beta}-i_{2\beta-1}\right)+(j-i_{2\nu-1}+1)\right]+
		\left[\left(\sum\limits_{\beta=1}^{m_1} i_{2\beta}-i_{2\beta-1}\right)+m_1\right]\right)$$
		times $M_1=\det(A^P_{j,k}(\overline{i_{2\nu}+1},\underline{i_{2\nu}};k,k+1))$, and 
		$\det(A^P_{j,k}(\overline{i_{2\nu}+1},\underline{i_{2\nu}};k,k+1))$ is equal to the product
		\begin{enumerate}
			\item $\text{sgn}\left(\sum\limits_{\ell=1}^{\nu-1}
			\left(\left[\sum\limits_{\beta=1}^n i_{2\beta}-i_{2\beta-1}\right]+
			\left[\sum\limits_{\beta=1}^\ell i_{2\beta}-i_{2\beta-1}\right]+1\right)\right),$
			\item $\text{sgn}\left(\left[\left(\sum\limits_{\beta=1}^n i_{2\beta}-i_{2\beta-1}\right)\right]+\left[\left(\sum\limits_{\beta=1}^{\nu-1} i_{2\beta}-i_{2\beta-1}\right)+j-i_{2\nu-1}+1\right]\right)$ and
			\item $\text{sgn}\left(\sum\limits_{\substack{\ell=\nu\\\ell\neq m_1}}^{n}
			\left(
			\left[\sum\limits_{\beta=1}^n i_{2\beta}-i_{2\beta-1}\right]+
			\left[\sum\limits_{\beta=1}^\ell i_{2\beta}-i_{2\beta-1}\right]
			\right)\right).$
		\end{enumerate}
		Thus $$s_1r_1M_1x(i_{2m_1},k)x(j+1,i_{2m_1}+1)=s_1r_1M_1\varphi_2^{(j,k)}(m_0,m_1)\varphi_2^{(j,k)}(\widetilde{(m_0,m_1)})$$
		is a term in $\det(A^P_{j,k})$.

		\begin{quote}
			Returning to our running example, let $m_1=2$.
			The cofactor of $\varphi_2^{(2,3)}(3,2)=x(9,12)$ is the product of
			\begin{enumerate}
				\item $s^1=\text{sgn}\left(\left[\left(\sum\limits_{\beta=1}^{2} i_{2\beta}-i_{2\beta-1}\right)+1\right]+
				\left[\left(\sum\limits_{\beta=1}^3 i_{2\beta}-i_{2\beta-1}\right)+3+1\right]\right)$
				\item $\varphi_2^{(2,3)}(\widetilde{(3,2)})=x(5,8),$
				\item $r^1=\text{sgn}\left(\left[\left(\sum\limits_{\beta=1}^{4} i_{2\beta}-i_{2\beta-1}\right)+1\right]+
				\left[\left(\sum\limits_{\beta=1}^{2} i_{2\beta}-i_{2\beta-1}\right)+2+1\right]\right)$ and
				\item $M^1=\det(A^P_{5,12}(\underline{5},\overline{9};8,12)).$
			\end{enumerate}
			It is not difficult to see that $M^1$ is equal to the the product
			\begin{enumerate}
				\item $\text{sgn}\left(\sum\limits_{\ell=1}^{2}
				\left(\left[\sum\limits_{\beta=1}^4 i_{2\beta}-i_{2\beta-1}\right]+
				\left[\left(\sum\limits_{\beta=1}^\ell i_{2\beta}-i_{2\beta-1}\right)+1\right]\right)\right)$ and
				\item $\text{sgn}\left(\sum\limits_{\ell=3}^{4}
				\left(\left[\sum\limits_{\beta=1}^4 i_{2\beta}-i_{2\beta-1}\right]+
				\left[\sum\limits_{\beta=1}^\ell i_{2\beta}-i_{2\beta-1}\right]\right)\right).$
			\end{enumerate}
			The cofactor of $\varphi_1^{(2,3)}(3,2)=x(6,11)$ is the product of
			\begin{enumerate}
				\item $s_1=\text{sgn}\left(\left[\left(\sum\limits_{\beta=1}^4 i_{2\beta}-i_{2\beta-1}\right)+2+1\right]+
				\left[\left(\sum\limits_{\beta=1}^3 i_{2\beta}-i_{2\beta-1}\right)+3+1\right]\right)$,
				\item $\varphi_1^{(2,3)}(\widetilde{(3,2)})=x(6,7)$,
				\item $r_1=\text{sgn}\left(\left[\left(\sum\limits_{\beta=1}^{1} i_{2\beta}-i_{2\beta-1}\right)+2\right]+
				\left[\left(\sum\limits_{\beta=1}^{2} i_{2\beta}-i_{2\beta-1}\right)+2\right]\right)$
				\item $M_1=\det(A^P_{5,11}(\overline{8},\underline{6};11,12)).$
			\end{enumerate}
			It is not difficult to show that $M_1$ is the product of 
			\begin{enumerate}
				\item $\text{sgn}\left(\sum\limits_{\ell=1}^{1}
				\left(\left[\sum\limits_{\beta=1}^4 i_{2\beta}-i_{2\beta-1}\right]+
				\left[\sum\limits_{\beta=1}^\ell i_{2\beta}-i_{2\beta-1}\right]+1\right)\right),$
				\item $\text{sgn}\left(\left[\left(\sum\limits_{\beta=1}^4 i_{2\beta}-i_{2\beta-1}\right)
				\right]+\left[\left(\sum\limits_{\beta=1}^{1} i_{2\beta}-i_{2\beta-1}\right)+2\right]\right)$ and
				\item $\text{sgn}\left(\sum\limits_{\substack{\ell=3}}^{4}
				\left(\left[\sum\limits_{\beta=1}^4 i_{2\beta}-i_{2\beta-1}\right]+
				\left[\sum\limits_{\beta=1}^\ell i_{2\beta}-i_{2\beta-1}\right]\right)\right).$
			\end{enumerate}
		\end{quote}
		Returning to the proof, notice that 
		\begin{align*}
			&\dfrac{s_1}{r_0}=\text{sgn}\left(m_1+1-\nu\right),\\  
			&r_1=\text{sgn}\left(\left[\left(\sum\limits_{\beta=1}^{\nu-1} i_{2\beta}-i_{2\beta-1}\right)+(j-i_{2\nu-1}+1)\right]+
			\left[\left(\sum\limits_{\beta=1}^{m_1} i_{2\beta}-i_{2\beta-1}\right)+m_1\right]\right)\text{ and}\\
			&\dfrac{M_0}{M_1}=
			\text{sgn}\left(-(\nu-1)-\left(\sum\limits_{\beta=1}^{\nu-1}i_{2\beta}-i_{2\beta-1}\right)+(j-i_{2\nu-1}+1)-
			\left[\left(\sum\limits_{\beta=1}^{m_1}i_{2\beta}-i_{2\beta-1}\right)-2(n-(\mu-1))\right]\right).
		\end{align*}
		Thus, $$\dfrac{s_1}{r_0}\cdot r_1 \cdot \dfrac{M_0}{M_1}=1.$$
		In other words, each term in $r_0M_0\varphi_2^{(j,k)}(\widetilde{(m_0)})+s_1r_1M_1\varphi_2^{(j,k)}(\widetilde{(m_0,m_1)})$
		has the same sign.
		Similarly we can show that each term in $r^0M^0\varphi_2^{(j,k)}(\widetilde{(m_0)})+s^1r^1M^1\varphi_2^{(j,k)}(\widetilde{(m_0,m_1)})$
		has the same sign.
		
		Thus $\det(A^P_{j,k})$ contains
		$$\pm \left(\varphi_{1}^{(j,k)}(\mu)\varphi_{1}^{(j,k)}(\widetilde{(\mu)})+\varphi_{1}^{(j,k)}(\mu,m_1)\varphi_{1}^{(j,k)}(\widetilde{(\mu,m_1)})-
		\varphi_{2}^{(j,k)}(\mu)\varphi_{2}^{(j,k)}(\widetilde{(\mu)})-\varphi_{2}^{(j,k)}(\mu,m_1)\varphi_{2}^{(j,k)}(\widetilde{(\mu,m_1)})\right).$$
		
		The rest follows from induction.
	\end{proof}
	
	\begin{case}
		Assume that $\mathcal{A}(P)\neq\emptyset.$ Then
		\label{det}
		$\det(A^P_{j,k})=\pm \widehat{D^P_{j,k}}$.
	\end{case}
	\begin{proof}
		Suppose that $\mathcal{A}(P)=\{i_{2\beta_1}<\dots<i_{2\beta_s}\}.$ 
		Define $$\mathcal{A}^P_\alpha=\begin{cases}
			|\{\ell:i_{2\ell} \notin \mathcal{A}(P) \text{ and }
			\ell > \alpha\}| &\text{ if $\alpha < n$ and $i_{2\alpha} \in \mathcal{A}(P)$}\\
			0 &\text{ otherwise.}
		\end{cases}$$
		
		Then the cofactor of $x_{j,k}$ is
		$$\text{sgn}\left(\left[\left(\sum\limits_{\beta=1}^n i_{2\beta}-i_{2\beta-1}\right)+\nu\right]+
		\left[\left(\sum\limits_{\beta=1}^\mu i_{2\beta}-i_{2\beta-1}\right)+\mu+1\right]\right)\det\left(A^P_{j,k}(\underline{j};k+1)\right),$$
		as before. 
		However $\det\left(A^P_{j,k}(\underline{j};k+1)\right)$ is different. 
		In this case, $\det\left(A^P_{j,k}(\underline{j};k+1)\right)$ is equal to
		$$\text{sgn}\left(\sum\limits_{\ell=1}^n\left(\left[\left(\sum\limits_{\beta=1}^n i_{2\beta}-i_{2\beta-1}\right)+1+\mathcal{A}^P_\ell\right]+
		\left[\left(\sum\limits_{\beta=1}^\ell i_{2\beta}-i_{2\beta-1}\right)+1\right]\right)\right).$$
		
		The cofactor of $y(i_{2\beta_r}+1,k+1)$ for some $r \in [1,s]$ is 
		$$\text{sgn}\left(\left[\left(\sum\limits_{\beta=1}^n i_{2\beta}-i_{2\beta-1}\right)+(n+1)-|\mathcal{A}(P)|+r\right]+
		\left[\left(\sum\limits_{\beta=1}^\mu i_{2\beta}-i_{2\beta-1}\right)+\mu+1\right]\right)$$
		times $\det\left(A^P_{j,k}((i_{2\beta}+1)';k+1)\right).$
		The reader can readily see that $A^P_{j,k}((i_{2\beta}+1)';k+1)=A^{P'}_{j,i_{2\beta}}$ where
		$P' \in \mathcal{I}_{n-1}$ is obtained from $P$ by removing the entries $(i_{2\beta-1},i_{2\beta})$
		for every $i_{2\beta}\in \mathcal{A}(P)$. 
		We have already shown that $\det(A^{P'}_{j,i_{2\beta}})=\pm D^{P'}_{j,i_{2\beta}}$ if $\mathcal{L}^{P'}(j,i_{2\beta}) \neq \emptyset$,
		and $\det(A^{P'}_{j,i_{2\beta}})=0$ if $\mathcal{L}^{P'}(j,i_{2\beta})=\emptyset$.
		Thus,
		$$\pm y(i_{2\beta_r}+1,k+1)D^{P'}_{j,i_{2\beta}}$$
		is a term of $\det(A_{j,k}^P)$.
		
		\begin{quote}
			Consider the case where $P=(0,2,3,3,4,6,8,10,12,14)$ and $(j,k)=(5,11)$.
			Let $P'=(0,2,4,6,8,10,12,14)$.
			Then
			$$A^P_{5,11}={\tiny
				\left(
				\begin{array}{ccccccccccccccccc}
					1 & 0 & x_{1,3} & x_{1,4} & 0       & 0 & x_{1,7} & 0       & 0       & x_{1,11}  & x_{1,12}  & 0         & 0        & x_{1,15}  \\
					0 & 1 & x_{2,3} & x_{2,4} & 0       & 0 & x_{2,7} & 0       & 0       & x_{2,11}  & x_{2,12}  & 0         & 0        & x_{2,15}  \\
					0 & 0 & 0       & 0       & 1       & 0 & x_{5,7} & 0       & 0       & x_{5,11}  & x_{5,12}  & 0         & 0        & x_{5,15}  \\
					0 & 0 & 0       & 0       & 0       & 1 & x_{6,7} & 0       & 0       & x_{6,11}  & x_{6,12}  & 0         & 0        & x_{6,15}  \\
					0 & 0 & 0       & 0       & 0       & 0 & 0       & 1       & 0       & x_{9,11}  & x_{9,12}  & 0         & 0        & x_{9,15}  \\
					0 & 0 & 0       & 0       & 0       & 0 & 0       & 0       & 1       & x_{10,11} & x_{10,12} & 0         & 0        & x_{10,15} \\
					0 & 0 & 0       & 0       & 0       & 0 & 0       & 0       & 0       & 0         & 0         & 1         & 0        & x_{13,15} \\
					0 & 0 & 0       & 0       & 0       & 0 & 0       & 0       & 0       & 0         & 0         & 0         & 1        & x_{14,15} \\
					0 & 0 & 1       & x_{2,3} & x_{2,4} & 0 & 0       & x_{2,8} & 0       & 0         & x_{2,11}  & x_{2,12}  & 0        & 0         \\
					0 & 0 & 0       & 0       & 0       & 1 & 0       & x_{5,8} & 0       & 0         & x_{5,11}  & x_{5,12}  & 0        & 0         \\
					0 & 0 & 0       & 0       & 0       & 0 & 1       & x_{6,8} & 0       & 0         & x_{6,11}  & x_{6,12}  & 0        & 0         \\
					0 & 0 & 0       & 0       & 0       & 0 & 0       &       0 & 0       & 1         & x_{10,11} & x_{10,12} & 0        & 0         \\
					0 & 0 & 0       & 0       & 0       & 0 & 0       &       0 & 0       & 0         & 0         & 0         & 0        & 1         \\
					0 & 0 & 0       & 1       & 0       & 0 & 0       &       0 & 0       & 0         & y_{4,12}  & 0         & 0        & 0         \\
				\end{array}\right)}.$$
			By taking the cofactor expansion down the column labeled by $k+1$, 
			we see that the cofactor $y(4,12)$ is the determinant of 
			$$A^P_{5,11}(4';12)={\tiny
				\left(
				\begin{array}{ccccccccccccccccc}
					1 & 0 & x_{1,3} & x_{1,4} & 0       & 0 & x_{1,7} & 0       & 0       & x_{1,11}  & 0         & 0        & x_{1,15}  \\
					0 & 1 & x_{2,3} & x_{2,4} & 0       & 0 & x_{2,7} & 0       & 0       & x_{2,11}  & 0         & 0        & x_{2,15}  \\
					0 & 0 & 0       & 0       & 1       & 0 & x_{5,7} & 0       & 0       & x_{5,11}  & 0         & 0        & x_{5,15}  \\
					0 & 0 & 0       & 0       & 0       & 1 & x_{6,7} & 0       & 0       & x_{6,11}  & 0         & 0        & x_{6,15}  \\
					0 & 0 & 0       & 0       & 0       & 0 & 0       & 1       & 0       & x_{9,11}  & 0         & 0        & x_{9,15}  \\
					0 & 0 & 0       & 0       & 0       & 0 & 0       & 0       & 1       & x_{10,11} & 0         & 0        & x_{10,15} \\
					0 & 0 & 0       & 0       & 0       & 0 & 0       & 0       & 0       & 0         & 1         & 0        & x_{13,15} \\
					0 & 0 & 0       & 0       & 0       & 0 & 0       & 0       & 0       & 0         & 0         & 1        & x_{14,15} \\
					0 & 0 & 1       & x_{2,3} & x_{2,4} & 0 & 0       & x_{2,8} & 0       & 0         & x_{2,12}  & 0        & 0         \\
					0 & 0 & 0       & 0       & 0       & 1 & 0       & x_{5,8} & 0       & 0         & x_{5,12}  & 0        & 0         \\
					0 & 0 & 0       & 0       & 0       & 0 & 1       & x_{6,8} & 0       & 0         & x_{6,12}  & 0        & 0         \\
					0 & 0 & 0       & 0       & 0       & 0 & 0       &       0 & 0       & 1         & x_{10,12} & 0        & 0         \\
					0 & 0 & 0       & 0       & 0       & 0 & 0       &       0 & 0       & 0         & 0         & 0        & 1         \\
				\end{array}\right)}.$$
			By taking the cofactor expansion down the column labeled by $i_4+1=4$, 
			we see that the cofactor of $x_{2,3}$ is $\text{sgn}(9+3)\det(A^P_{5,11}(4';12))=0$.
			Also notice that $\mathcal{L}^{P}(5,4)=\emptyset$.
			Thus $\widehat{D^P}_{5,4}=D^{P'}_{5,4}.$
			
			If $P=(0,2,4,6,7,7,8,10,12,14)$ and $(j,k)=(5,11)$, then
			$$A^P_{5,11}={\tiny
				\left(
				\begin{array}{ccccccccccccccccc}
					1 & 0 & x_{1,3} & 0       & 0 & x_{1,7} & x_{1,8} & 0       & 0       & x_{1,11}  & x_{1,12}  & 0         & 0        & x_{1,15}  \\
					0 & 1 & x_{2,3} & 0       & 0 & x_{2,7} & x_{2,8} & 0       & 0       & x_{2,11}  & x_{2,12}  & 0         & 0        & x_{2,15}  \\
					0 & 0 & 0       & 1       & 0 & x_{5,7} & x_{5,8} & 0       & 0       & x_{5,11}  & x_{5,12}  & 0         & 0        & x_{5,15}  \\
					0 & 0 & 0       & 0       & 1 & x_{6,7} & x_{6,8} & 0       & 0       & x_{6,11}  & x_{6,12}  & 0         & 0        & x_{6,15}  \\
					0 & 0 & 0       & 0       & 0 & 0       &       0 & 1       & 0       & x_{9,11}  & x_{9,12}  & 0         & 0        & x_{9,15}  \\
					0 & 0 & 0       & 0       & 0 & 0       &       0 & 0       & 1       & x_{10,11} & x_{10,12} & 0         & 0        & x_{10,15} \\
					0 & 0 & 0       & 0       & 0 & 0       &       0 & 0       & 0       & 0         & 0         & 1         & 0        & x_{13,15} \\
					0 & 0 & 0       & 0       & 0 & 0       &       0 & 0       & 0       & 0         & 0         & 0         & 1        & x_{14,15} \\
					0 & 0 & 1       & x_{2,4} & 0 & 0       & x_{2,7} & x_{2,8} & 0       & 0         & x_{2,11}  & x_{2,12}  & 0        & 0         \\
					0 & 0 & 0       & 0       & 1 & 0       & x_{5,7} & x_{5,8} & 0       & 0         & x_{5,11}  & x_{5,12}  & 0        & 0         \\
					0 & 0 & 0       & 0       & 0 & 1       & x_{6,7} & x_{6,8} & 0       & 0         & x_{6,11}  & x_{6,12}  & 0        & 0         \\
					0 & 0 & 0       & 0       & 0 & 0       & 0       &       0 & 0       & 1         & x_{10,11} & x_{10,12} & 0        & 0         \\
					0 & 0 & 0       & 0       & 0 & 0       & 0       &       0 & 0       & 0         & 0         & 0         & 0        & 1         \\
					0 & 0 & 0       & 0       & 0 & 0       & 1       &       0 & 0       & 0         & y_{8,12}  & 0         & 0        & 0         \\
				\end{array}\right)}.$$
			Notice that the cofactor of $y_{8,12}$ is $\text{sgn}(14+11)\det(A^P_{5,11}(8';12)=-\det(A^{P'}_{5,7})$.
		\end{quote}
		In order to see that $D_{j,k}^{P'}$ and $y(i_{2\beta_r}+1,k+1)D_{j,i_{2\beta_r}}^{P'}$ have opposite signs,
		it is enough to show that the sign of $x_{j,k}$ and of $y(i_{2\beta}+1,k+1)x(j,i_{2\beta_r})$ in $\widehat{D^P_{j,k}}$
		have opposite signs.
		
		From our previous work we know that the coefficient of $x(j,i_{2\beta_r})$ in $D^{P'}_{j,i_{2\beta_r}}$ is
		$$\text{sgn}\left(\left[\left(\sum\limits_{\beta=1}^n i_{2\beta}-i_{2\beta-1}\right)+\nu\right]
		+\left[\left(\sum\limits_{\beta=1}^{\beta_r} i_{2\beta}-i_{2\beta-1}\right)+\beta_r\right]\right)
		\det\left(A_{j,i_{2\beta_r}}^{P'}(\underline{j};i_{2\beta_r})\right)$$
		where $\det\left(A_{j,i_{2\beta_r}}^{P'}(\underline{j};i_{2\beta_r})\right)$ is
		$$\text{sgn}\left(\sum\limits_{\ell\neq \beta_r}\left(\left[\left(\sum\limits_{\beta=1}^n i_{2\beta}-i_{2\beta-1}\right)+1+\mathcal{A}^P_\ell\right]+
		\left[\left(\sum\limits_{\beta=1}^\ell i_{2\beta}-i_{2\beta-1}\right)+1\right]\right)\right).$$
		Then, by multiplying together and simplifying, we get
		$$\text{sgn}(n-|\mathcal{A}(P)|+r+\beta_r+\mathcal{A}_{\beta_r}^P+1).$$
		Note that $\mathcal{A}^P_{\beta_r}=|\{i_{2\ell}:\ell > \beta_r\}| + |\{i_{2s} \in \mathcal{A}(P): s > r\}|.$
		Then
		$$\text{sgn}(n-|\mathcal{A}(P)|+r+\beta_r+\mathcal{A}_{\beta_r}^P+1)=\text{sgn}(1)=-1.$$
		Thus Lemma \ref{det} is proved.
	\end{proof}

	\subsection{Proof of $\mathcal{J}(P)\subseteq\mathcal{K}(P)$}
	
	\begin{proof}
		Let $S_r=\{\underline{\gamma}: 1 \leq \beta \leq n \text{ and } i_{2\beta-1}+1 \leq \gamma \leq i_{2\beta}-1\}$ and 
		$S_c=\{\gamma: 1 \leq \beta \leq n \text{ and } i_{2\beta}+2 \leq \gamma \leq i_{2\beta+1}\}\cup\{m-2\}$.
		
		Let $S_r',S_c' \subseteq [1,m-3]$ such that $|S_r'|=|S_r|-1$ and $|S_c'|=|S_c|-1$.
		Notice that $|S_r|=\left(\sum\limits_{\ell=1}^n i_{2\ell}+i_{2\ell-1}\right)-(n-|\mathcal{A}(P)|)$
		and $|S_c|=\left(\sum\limits_{\ell=1}^n i_{2\ell+1}-i_{2\ell}\right)-n+1$.
		Define $N_1(S_r';S_c')$ to be the $e_1(P)\times e_1(P)$ submatrix obtained from 
		the $\left(2\left(\sum\limits_{\ell=1}^n i_{2\ell}+i_{2\ell-1}\right)+|\mathcal{A}(P)|\right) \times (m-2)$ matrix $N_1(P)$ 
		by removing the rows labeled by $S_r'$ and the columns labeled by $S_c'$.
		
		Using the standard method of finding (quadratic) relations between minors, we get
		$$\det\left(N_1(S_r';S_c')\right)=
		\sum\limits_{(j,k) \in JK}\pm 
		\det\left(N_1(S_r-\{\underline{j}\};S_c-\{k+1\})\right)\det\left(N_1(S_r'\cup \{\underline{j}\};S_c'\cup \{k+1\})\right).$$
		Notice that $A^P_{j,k}=N_1(S_r-\{\underline{j}\};S_c-\{k+1\})$.
		Thus every $(e_1(P)+1)\times(e_1(P)+1)$ minor of $N_1(P)$ can be generated by $\{\widehat{D_{j,k}^{P}}:(j,k) \in JK\}$.
	\end{proof}

	\begin{example}
		Consider $m-3=7$ and $P=(0,2,4,6)$.
		In this case, $N_1(P)$ is
		$${\tiny
			\left(
			\begin{array}{ccccccccccccccccc}
				1 & 0 & x_{1,3} & x_{1,4} & 0       & 0 & x_{1,7} & 0\\
				0 & 1 & x_{2,3} & x_{2,4} & 0       & 0 & x_{2,7} & 0\\
				0 & 0 & 0       & 0       & 1       & 0 & x_{5,7} & 0\\
				0 & 0 & 0       & 0       & 0       & 1 & x_{6,7} & 0\\
				0 & 1 & 0       & x_{1,3} & x_{1,4} & 0 & 0       & x_{1,7} \\
				0 & 0 & 1       & x_{2,3} & x_{2,4} & 0 & 0       & x_{2,7} \\
				0 & 0 & 0       & 0       & 0       & 1 & 0       & x_{5,7} \\
				0 & 0 & 0       & 0       & 0       & 0 & 1       & x_{6,7} \\
			\end{array}\right)}.
		$$
		Then
		\begin{align*}
			\det\left(N_1(\underline{6};1)\right)=&(-x_{5,7}-x_{6,7}^2)(x_{1,4}-x_{1,3}x_{2,3})(x_{2,3}x_{2,4}x_{5,7}+x_{1,4}x_{5,7}+x_{2,7})\\
			&+(x_{1,7}+x_{2,3}x_{2,7}+x_{2,7}x_{6,7}+x_{1,4}x_{5,7}x_{6,7}+x_{2,3}x_{2,4}x_{5,7}x_{6,7})(-x_{1,4}x_{6,7}+x_{1,3}x_{2,3}x_{6,7})\\
			&-(-x_{1,3}+x_{2,4}-x_{2,3}^2)(-x_{1,3}x_{2,7}x_{6,7}+x_{2,4}x_{5,7}^2+x_{1,7}x_{5,7})\\
			=&\det\left(N_1(\underline{1};4)\right)\det\left(N_1(\underline{5},\underline{6};1,8)\right)
			+\det\left(N_1(\underline{5};4)\right)\det\left(N_1(\underline{1},\underline{6};1,8)\right)\\
			&-\det\left(N_1(\underline{5};8)\right)\det\left(N_1(\underline{1},\underline{6};1,4)\right)\\
			=&\det\left(A_{5,7}\right)\det\left(N_1(\underline{5},\underline{6};1,8)\right)
			+\det\left(A_{1,7}\right)\det\left(N_1(\underline{1},\underline{6};1,8)\right)\\
			&-\det\left(A_{1,3}\right)\det\left(N_1(\underline{1},\underline{6};1,4)\right).
		\end{align*}
	\end{example}

\end{document}